%% file: stokes_new.tex
\documentclass[a4paper]{article}
\usepackage{amsmath,amssymb}
\usepackage[colorlinks,citecolor=blue]{hyperref}
\usepackage[utf8]{inputenc}
\usepackage{tikz}
\usetikzlibrary{arrows}
\usetikzlibrary{decorations.pathreplacing,decorations.markings}
\usepackage[english]{babel}

\newtheorem{theorem}{Theorem}[section] 
\newtheorem{proposition}[theorem]{Proposition} 
\newtheorem{conjecture}[theorem]{Conjecture} 
 
\newtheorem{lemma}[theorem]{Lemma}

\newenvironment{proof}{\begin{trivlist}\item{\bf{Proof.}}}
  {\hfill\rule{2mm}{2mm}\end{trivlist}}

\newcommand{\ZZ}{\mathbb{Z}}

\newcommand{\carre}{\Box}
\newcommand{\rk}{\operatorname{rk}}
\newcommand{\op}{\operatorname{op}}
\newcommand{\SN}{\mathrm{SN}}

\newcommand{\TA}{\mathbb{A}}
\newcommand{\TB}{\mathbb{B}}
\newcommand{\TD}{\mathbb{D}}

\newcommand{\T}{\mathbf{T}}

\renewcommand{\mod}{\operatorname{mod}}

\title{Stokes posets and serpent nests}
\author{F. Chapoton\footnote{L'auteur a bénéficié d'une aide de
    l'Agence Nationale de la Recherche (projet Carma, référence
    ANR-12-BS01-0017). Il remercie aussi l'Institut Mittag-Leffler pour son accueil chaleureux et ses excellentes conditions de travail.}}
\date{\today}
\begin{document}

\maketitle

\begin{abstract}
  We study two different objects attached to an arbitrary
  quadrangulation of a regular polygon. The first one is a poset,
  closely related to the Stokes polytopes introduced by
  Baryshnikov. The second one is a set of some paths configurations
  inside the quadrangulation, satisfying some specific
  constraints. These objects provide a generalisation of the existing
  combinatorics of cluster algebras and nonnesting partitions of type
  $\TA$.
\end{abstract}

{\bf keywords:} poset, quadrangulation, Tamari lattice


\section*{Introduction}

The research leading to this article started with the desire to
understand the apparition of associahedra in an article of Baryshnikov
\cite{baryshnikov} and in an article of Kapranov and Saito
\cite{kapranov_saito}. Associahedra, also known as Stasheff polytopes,
are classical in algebraic topology, where they are used to define and
study associativity up to homotopy \cite{stasheff}. More recently,
they have made their way into the theory of cluster algebras, in which
they control the combinatorial aspects of the type $\TA$ cluster
algebras \cite{fominz_Y}. It was not clearly obvious how to relate the
contexts used by Baryshnikov and Kapranov-Saito to either
associativity or cluster algebra theory. Moreover both articles were
seeing associahedra as members of a larger family of polytopes, all
closely tied together, whereas associahedra usually stand alone.

Another strand of research then came into play, centred on the Tamari
lattices \cite{tamfest}. The Tamari lattices are closely related to
associahedra and associativity, and can be defined by orienting the
associativity relation from left-bracketing to right-bracketing. They
are also well-understood in the context of cluster algebras, where
they are just one among several Cambrian lattices of type $\TA$,
depending on the choice of an orientation of the Dynkin
diagram \cite{reading, reading_speyer}. It is remarkable that all the
Cambrian lattices attached to type $\TA_n$ share the same underlying
graph, which is nothing else that the graph of vertices and edges of the
associahedra. This is also the flip graph of triangulations and the
mutation graph of cluster algebras of type $\TA$.

The Tamari lattices have yet another, maybe less well-known,
interpretation. They can be seen as a special case of the partial
orders defined by Happel-Unger \cite{happel_unger} and Riedtmann-Schofield \cite{riedtmann} on the sets of
tilting modules over quivers.

One can consider the notion of module over the Tamari lattices,
defined for example using their incidence algebras. It turns out that
there is a very interesting subset of modules, indexed by the set of
quadrangulations of a regular polygon. It would be too long to tell
here how precisely these modules were stumbled upon, but it involved
an anticyclic operad made with all Tamari lattices, and a mysterious
map from the $K_0$ of Tamari lattices to rational functions
\cite{chap_mould}. All this seems in fact to take place naturally in
the Tamari lattices seen as partial orders on tilting modules.

Anyway, these modules indexed by quadrangulations can be seen as
living in the derived categories of modules over the Tamari lattices. A
natural question to ask was then: what are the morphisms between these
modules ? It appeared that they may be described using flips of
quadrangulations, under the condition that a flip should not be
repeated twice.

At this point, the connection was made with the work of Baryshnikov,
which also involved quadrangulations and their flips. It turned out
that the match was perfect. This is essentially the story behind the
content of the article, which concentrates on the combinatorial side
and leaves most of the representation theory aspects aside.

Let us now describe what is done, and then comment on the results.

In \autoref{poset_section}, a poset is associated with every
quadrangulation in a regular polygon. This is done using the
compatibility relation between quadrangulations introduced by
Baryshnikov. First, one defines directed graphs (digraphs) by
orienting the allowed flips in a specific way. Then it takes some work
to show that these digraphs are Hasse diagrams of posets, that we
choose to call the Stokes posets. The main tool is an inductive
description of the digraphs, relating a quadrangulation and a smaller
one with a leaf square removed. One also shows that the associated
undirected graphs are connected and regular.

In the articles of Baryshnikov and his collaborators on the subject
\cite{baryshnikov, baryX, hickoketal}, no digraphs were considered,
but instead polytopes were associated with quadrangulations. Their
graphs of vertices and edges coincides by definition with the associated undirected
graphs of our digraphs.

In the rest of the first section, one proposes a conjectural manner
to define the same digraphs without having to use the compatibility
condition. The proof that this works is sadly missing.

In \autoref{poset_property_section}, one considers examples and
properties of the Stokes posets. In particular, one proves that the Tamari
lattice is recovered as a special case, for a very regular
quadrangulation.

It is then shown that the posets attached to some quadrangulations can
be written as Cartesian products of smaller posets of the same
kind. This happens under the condition that there is a bridge in the
quadrangulation. One then says a few words about the simplicial
complexes attached to the quadrangulation, which can be thought of as
the dual of Baryshnikov's Stokes polytopes. In particular, enumerative
aspects of this simplicial complex can be encoded into an $F$-triangle.

Next, a conjecture is proposed about what happens when a
quadrangulation is twisted along an edge, namely that the undirected graph,
simplicial complex and $F$-triangle should not change.

Then \autoref{exc_section} is a rather sketchy proposal for a vast
generalisation of the theory of Stokes posets to other root systems
and finite Coxeter groups. This is stated first in the setting of
exceptional sequences, and then using factorisation of Coxeter
elements into reflections, or equivalently maximal chains in the
noncrossing partition lattices.

In the next \autoref{nests_section}, one introduces combinatorial
configurations fitting inside quadrangulations. They are made of
elementary pieces called \textit{serpents}, and therefore named
\textit{serpent nests}. These combinatorial objects should be closely
related to the posets introduced in the first section. For every
quadrangulation $Q$, there should be as many elements in the poset
attached to $Q$ as there are serpent nests in $Q$. The serpent nests
should be thought of as playing the same role as nonnesting partitions
play for classical root systems.

After their definition, the section contains various results about
serpent nests, that sometimes are similar to some statements or
conjectures about the Stokes posets made in the previous sections.

First, it is shown that serpent nests do factorise in presence of a
bridge, just as Stokes posets do. Then an enumerative study is made,
by introducing $h$-vectors and $H$-triangles that counts serpent
nests according to their ranks. A duality is obtained, that implies a
symmetry of the $h$-vector. The $H$-triangle is conjectured to be
related to the $F$-triangle, by the same formula that work between
clusters $F$-triangles and nonnesting partitions $H$-triangles for all
root systems (see for example \cite{thiel_long} and
the references therein).

It is then shown that, for the same regular quadrangulation already
considered in the first section, the serpent nests are in a simple
bijection with nonnesting partitions of type $\TA$. Another
interesting family of examples is considered, for which the number of
serpent nests is given by a Lucas sequence. This gives essentially the
(bridge-less) quadrangulations with the smallest possible numbers of
serpent nests.

A paragraph is devoted to an interesting strategy to count serpent
nests, namely cut quadrangulations in two parts along one edge, compute something related to open serpent nests for
every half, and reconstitute the result from that. This seems also to be
related to the Dynkin type $\TB$ and the usual folding procedure of
simply-laced Dynkin diagrams.

Then \autoref{parabo_section} deals with an algebraic structure that
gathers all the quadrangulations together, namely a commutative
algebra endowed with a derivation. Product is disjoint union or union
along a bridge and the derivative is given by cutting at inner
edges. This is analogue to the parabolic structure on quivers and Dynkin diagrams.

The last \autoref{enum_section} contains elementary enumerative
results about various sorts of quadrangulations. The equivalence
classes under twisting are also counted, as they should correspond to
the distinct undirected graphs of Stokes posets and do describe
quadrangulations with the same serpent nests. For some reason related
to fans and not studied in this article, quadrangulations that do not
contain a cross are also considered.

\medskip

Let us now summarise the article and give further comments. Taking as
input a quadrangulation, one defines a poset and a finite set. For
some special quadrangulations, one recovers the Tamari lattices and
the nonnesting partitions associated to the Dynkin type $\TA$ by the
theory of cluster algebras and root systems. It is in fact expected that all the posets
will be lattices, and also that other Cambrian lattices can be
obtained for other ribbon quadrangulations.

It is therefore very tempting to make an analogy where
quadrangulations play the role of quivers, and quadrangulations up to
twist that of Dynkin diagrams. Then one gets analogues of clusters,
positive clusters, exchange graphs, but only at a combinatorial
level. It seems that most of the properties known in the combinatorial
side of cluster theory do extend. It would be interesting to see if
analogues of cluster variables could make sense.

As explained before, the article \cite{kapranov_saito} of Kapranov and
Saito was one source of motivation. The relationship with our results
is not as direct as in the case of Baryshnikov, but should exist. It
would probably be best seen using the noncrossing trees and
exceptional sequences point of view explained in
\autoref{exc_section}.

Flips in quadrangulations have been considered also in the context of
$m$-analogues of cluster theory for $m=2$. The $2$-Cambrian lattices
introduced recently in \cite{cataland} are defined by allowing two
consecutive flips inside an hexagon, but not three. The precise
relationships with this article remain to be explored.

As a side remark, one can note that very similar combinatorial objects
as the ones involved here (quadrangulations and noncrossing trees)
have appeared in the study of multiple zeta values
\cite{ecalle,dupont}. Whether there is a connection or not remains to
be settled.

The name of Stokes polytopes was used by Baryshnikov because of the
relation of his work with the Stokes phenomenon. In the article
\cite{naka}, the Stokes phenomenon is also connected to cluster
algebras. It is not clear if these connections are the same or at
least could fit in a common framework.

Let us finish by saying that there are several aspects missing in the
present article, the most important ones being the in-depth study of
related fans and polytopes, only sketched in \cite{baryshnikov}, and
also the closely related question to check that the posets are indeed
lattices. These questions are left for future work.

\section{Stokes posets}

\label{poset_section}

Informally, a \textit{quadrangulation} is a partition of the interior
of a regular polygon into quadrilaterals. This is possible if and only
if the regular polygon has an even number of vertices.

More formally, one considers a regular polygon with vertices numbered
from $0$ to $2n+1$ in the negative (clockwise) direction. A quadrangulation is
defined as a collection of pairs of (not consecutive) vertices of this
polygon, such that the associated edges do not cross and define only
quadrilateral regions.

\begin{figure}[ht]
  \begin{center}
    \includegraphics[height=4cm]{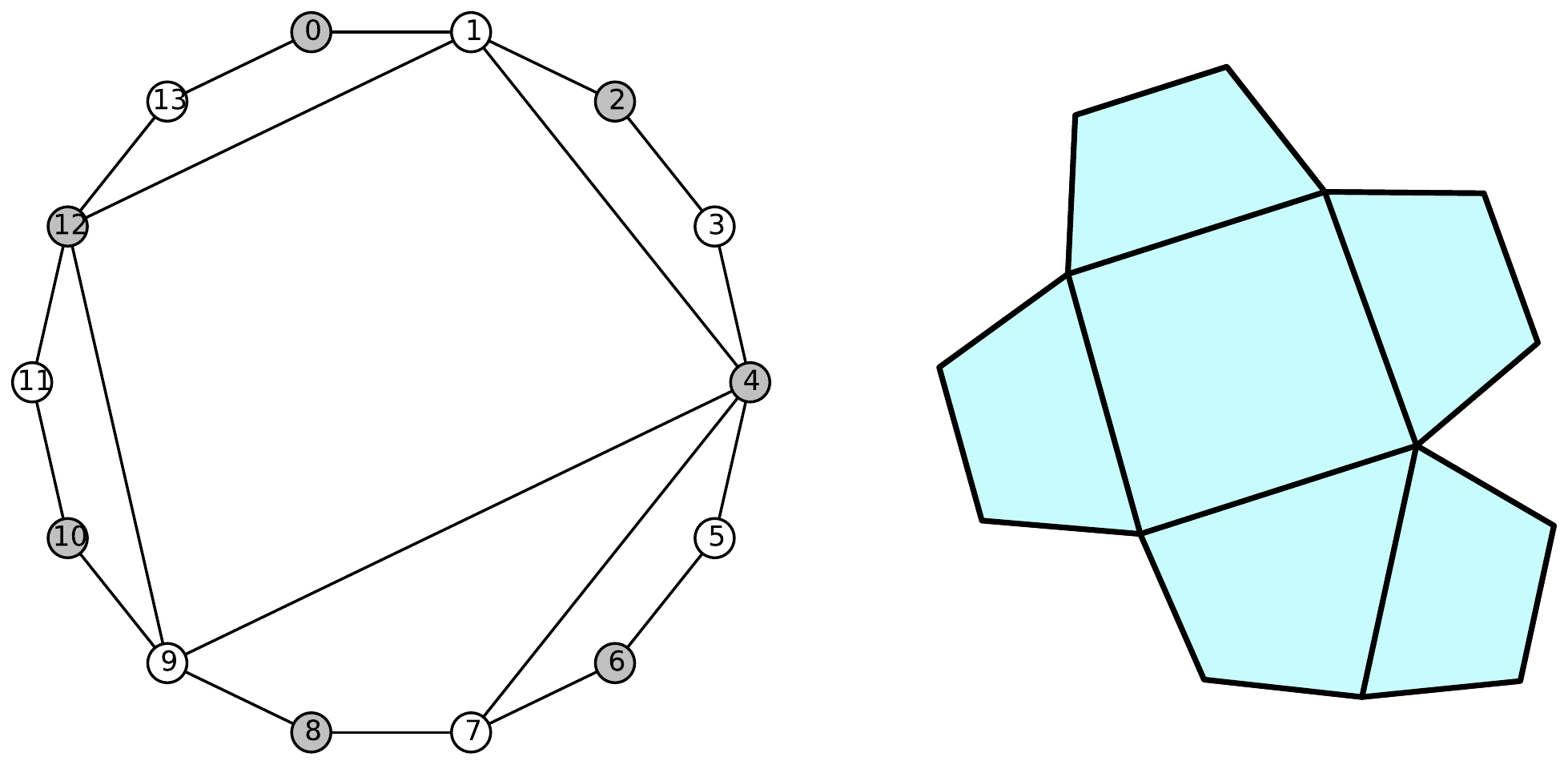}
    \caption{Two views of the same quadrangulation ($n=6$)}
    \label{expl_type}
  \end{center}
\end{figure}
It will sometimes be convenient to depict quadrangulations by
deforming the quadrilaterals so that they become more regular and of
similar area. The boundary is then no longer a regular polygon, but a
simple polygonal closed curve, as in the right of \autoref{expl_type}.

The number of distinct quadrangulations inside the polygon with $2n+2$
vertices is also the number of ternary trees with $2n+1$ leaves and is
classically given by
\begin{equation}
  \frac{1}{2n+1}\binom{3n}{n}.
\end{equation}
The bijection is the obvious planar duality, after a boundary edge has
been chosen as root.

Let us fix some terminology. A quadrangulation in the regular polygon
with $2n+2$ vertices is made of $2n+2$ \textit{boundary edges} and
$n-1$ \textit{inner edges}. The $n$ quadrilateral regions will be called
\textit{squares}. The word \textit{edge} will generally be used to
designate both inner and boundary edges, unless the context make it
clear that it must be of one kind only.

\subsection{Compatibility}

\label{compat}

Y. Baryshnikov has introduced a compatibility condition between
quadrangulations. Let us describe this precisely.

It will be convenient for us to consider a fixed quadrangulation $Q$
drawn inside a regular polygon with $2 n + 2$ vertices, and to define only
compatibility with $Q$. Let us choose an alternating black and white
colouring of the vertices of the regular polygon. This choice is
arbitrary, and the compatibility condition will not depend on it.

The quadrangulation $Q$ is seen as a (blue and dashed) theatre
backdrop, on which one superimposes another (red) regular polygon,
slightly rotated in the negative (clockwise) direction by an angle of
$\frac{2\pi}{4 n + 4}$, so that the vertices interlace. The
black-and-white colouring of the vertices of the rotated polygon is
inherited from the initial colouring.

\begin{figure}[ht]
  \begin{center}
    \includegraphics[height=5cm]{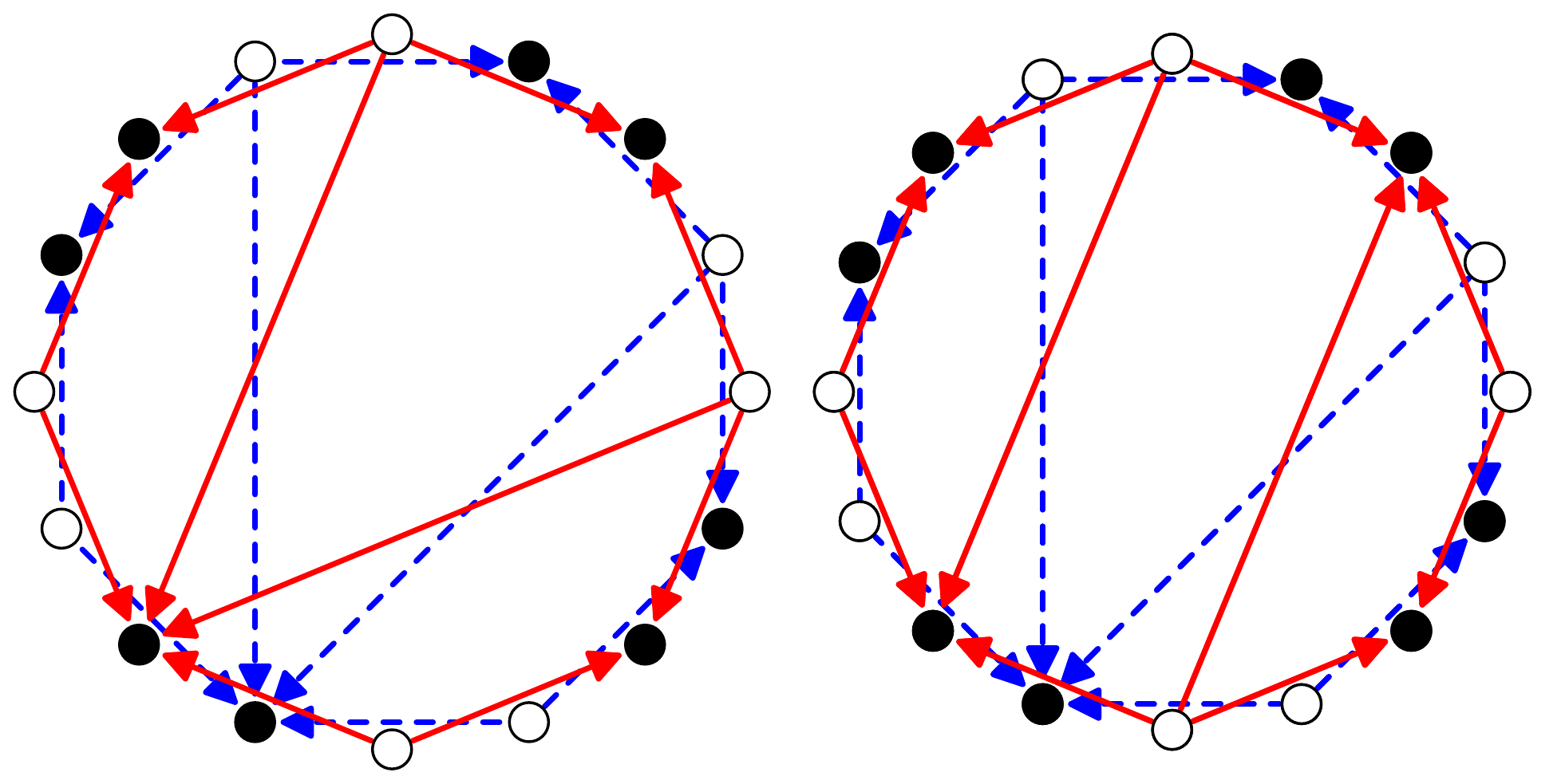}

    \caption{Backdrop $Q$ (blue and dashed) and two $Q$-compatible $Q'$ and $Q''$ (red)}
    \label{expl_comp}
  \end{center}
\end{figure}

By convention, all the edges (red or blue) are oriented from white vertices $\circ$ to black
vertices $\bullet$.

\medskip

An edge $i$ in the red polygon is said to be \textit{compatible} with an edge
$j$ of $Q$ if the pair (red oriented edge $i$, blue oriented edge $j$)
defines the positive orientation of the plane.

An edge $i$ in the red polygon is said to be $Q$-\textit{compatible}
if it is compatible with all edges of $Q$. This is equivalent to say
that $i$ is compatible with all inner edges of $Q$, because the
condition is always satisfied with respect to boundary edges of $Q$.

Moreover, compatibility with $Q$ is also automatic for the boundary edges
of the red polygon.

A quadrangulation $Q'$ is $Q$-\textit{compatible} if and only if all its
edges are $Q$-compatible. This is equivalent to require that every
inner edge of $Q'$ is compatible with all inner edges of $Q$.

Note that $Q$-compatibility is unchanged under switching the choice of the
black and white colouring of the vertices, as this changes the
orientation of all edges.

\medskip

It is clear that there is a finite number of $Q$-compatible
quadrangulations. This number depends on $Q$. There are always two
simple ones. First, $Q$ (seen in the red polygon after a negative
rotation by $\frac{2\pi}{4 n + 4}$) is clearly compatible with itself. The
other one is the quadrangulation obtained from the backdrop $Q$ by
positive rotation by $\frac{2\pi}{4 n + 4}$ (or from the red $Q$ by
positive rotation by $\frac{2\pi}{2 n + 2}$). Let us call it $\tau(Q)$\footnote{This notation is chosen to remind of the Auslander-Reiten translation.}.

\subsection{Flips}

\begin{proposition}
  \label{quad_flip}
  Let $Q$ be a fixed quadrangulation. Let $Q'$ be a $Q$-compatible
  quadrangulation. For every inner edge $i$ of $Q'$, there exists
  exactly one other inner edge $j$ such that $Q' \setminus \{i\}
  \cup \{j\}$ is a $Q$-compatible quadrangulation.
\end{proposition}
\begin{proof}
  This follows from Lemma \ref{hexa_flip} below.
\end{proof}

\begin{lemma}
  \label{hexa_flip}
  Let $Q$ be a fixed quadrangulation. Consider an hexagon $H$ made of
  six $Q$-compatible edges. Then there are exactly two $Q$-compatible
  inner edges in the interior of $H$.
\end{lemma}
\begin{proof}
  Let $a_0,b_0,a_1,b_1,a_2,b_2$ be the vertices of the hexagon $H$ in
  clockwise order, with indices in $\ZZ/3\ZZ$. One can assume without
  restriction that the colours are white for $a$ vertices and black for
  $b$ vertices. It follows from the $Q$-compatibility of $H$ that
  inner edges of $Q$ can only enter $H$ by crossing a inner edge
  $a_{i+1} \rightarrow b_{1}$ of $H$ and can only exit by crossing a
  inner edge $a_i \rightarrow b_i$ of $H$. If there is no inner edge
  in $Q$ entering $H$ through $a_{i+1} \rightarrow b_{i}$ and leaving $H$
  through $a_{i+2} \rightarrow b_{i+2}$ for some $i$ (call it a
  \textit{long inner edge}), then $Q$ contains an empty hexagon, which
  is absurd. Therefore there is at least a long inner edge in $Q$, say
  from $a_2 \rightarrow b_{1}$ to $a_{0} \rightarrow b_{0}$. All long
  inner edges must then enter and exit $H$ by crossing the same edges of $H$, for otherwise they would cross. Then among the three
  possible inner edges inside the hexagon $H$, one can check that
  $a_0 \rightarrow b_1$ and $a_2 \rightarrow b_0$ are $Q$-compatible,
  but that $a_1 \rightarrow b_2$ is not.
\end{proof}

\subsection{Flip graphs}

Proposition \ref{quad_flip} implies that one can move between
$Q$-compatible quadrangulations by removing any inner edge and
inserting another one. These moves between quadrangulations will be
called \textit{flips}.

This defines a graph with vertices the $Q$-compatible quadrangulations
and edges the flips. This is clearly a regular graph where every
vertex has arity $n-1$. Let us denote it by $\gamma_Q$ and call it the
\textit{flip graph} of $Q$-compatible quadrangulations. Later on, one
will show that this graph is connected.



\medskip

Let us now describe how one can orient the edges of the flip graph
$\gamma_Q$.

\begin{figure}[ht]
  \begin{center}
    \includegraphics[height=5cm]{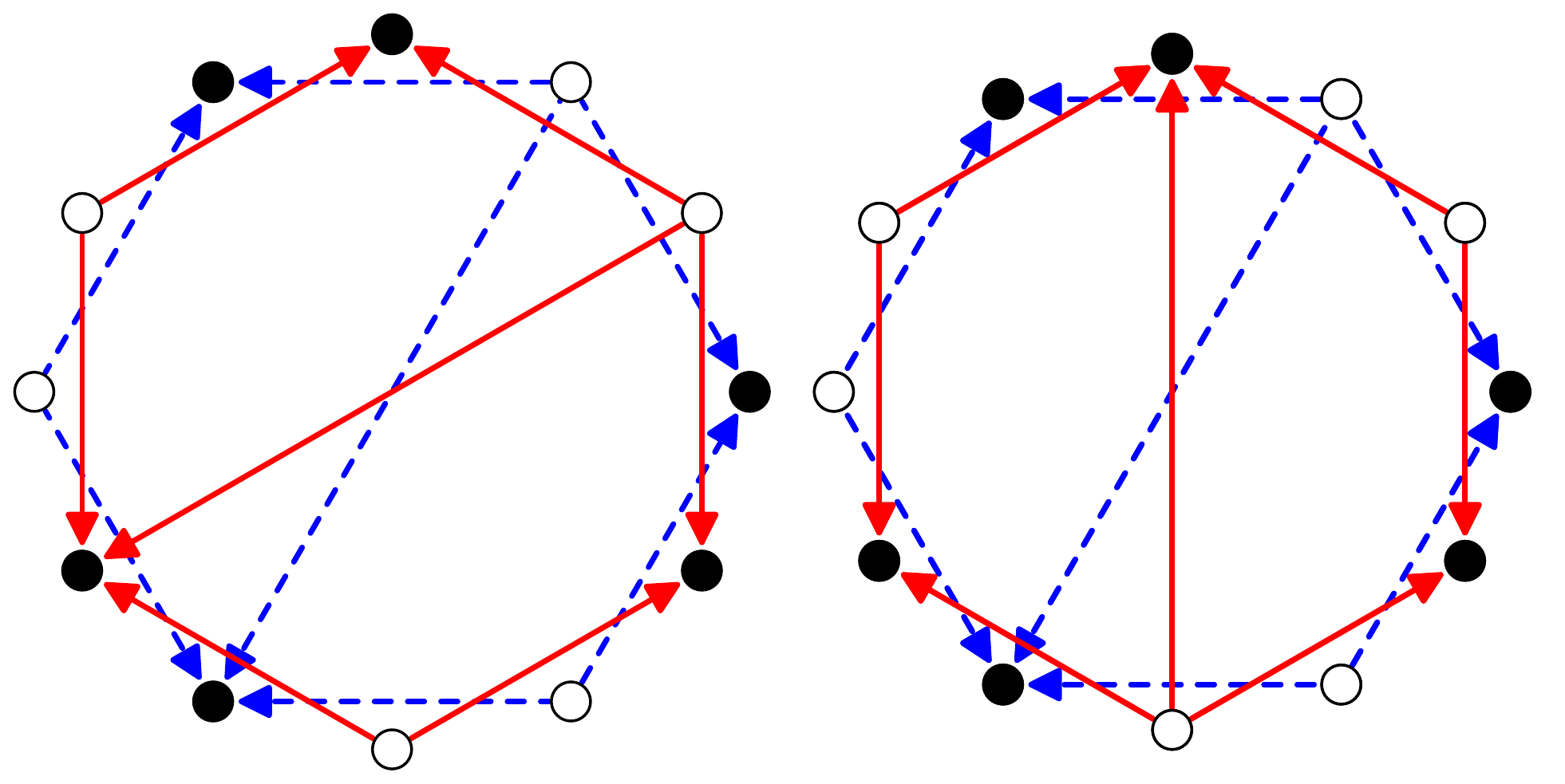}
    \caption{Simple example of flip, oriented from left \textbf{(id)} to right \textbf{(op)}}
    \label{flip}
  \end{center}
\end{figure}

Consider a flip relating two quadrangulations $Q'$ and
$Q''$. Only the inner edge inside the modified hexagon $H$ is
changed. The two possible inner edges $i'$ and $i''$ can be
distinguished as follows. Recall from the proof of Lemma
\ref{hexa_flip} that there must be at least one \textit{long inner edge} $j$ in $Q$
entering the hexagon and leaving it through the opposite side. The
starting vertex of the inner edge of $H$ is adjacent
\begin{description}
\item[(id)] either to the starting vertex of the long inner edge $j$,
\item[(op)] or to the ending vertex of the long inner edge $j$.
\end{description}

Let us orient the flips from the case of adjacent starting vertices \textbf{(id)} to the case of far-away starting vertices \textbf{(op)}.

This is illustrated in \autoref{flip}, with the \textbf{(id)} case on the
left and the \textbf{(op)} case on the right.

For an oriented flip from $Q'$ to $Q''$, one says that it is an
\textit{out-flip} for $Q'$ and an \textit{in-flip} for $Q''$. For
every $Q$-compatible quadrangulation $Q'$, the number of in-flips plus
the number of out-flips is $n - 1$, as every inner edge corresponds
either to an in-flip or to an out-flip.

This allows to define an orientation of the flip graph $\gamma_Q$. Let
us call this oriented graph $\Gamma_Q$.

\medskip

To study properties of the graph $\Gamma_Q$, one will proceed by induction on
the \textit{size} $n$ of the quadrangulation $Q$, which is its number of squares.

In any quadrangulation with $n\geq 2$ squares, one can always find a
boundary edge $e$ of $Q$ that bounds a square $s$ with exactly one
neighbour square\footnote{This is because a quadrangulation is a tree of squares, hence has at least one leaf.}. Up to switching black and white, one can assume that
the edge $e$ goes from white to black in counterclockwise order. Let
$s$ be the square bordering the chosen edge $e$ and let
$Q^{\setminus s}$ be the quadrangulation obtained from $Q$ by removing
$s$.

Let $Q'$ be any $Q$-compatible quadrangulation. The underlying red
polygon has two white vertices $i$ and $i+2$ (modulo $2n+2$) that lie
just before and just after the chosen blue edge $e$. Recall that the
vertices are numbered in clockwise order.

One can associate with $Q'$ a $Q^{\setminus s}$-compatible
quadrangulation $\theta_s(Q')$ as follows. Pictorially, one collapses
the boundary section of the red polygon between white vertices $i$ and
$i+2$ to a single white vertex in a smaller red polygon. Edges of $Q'$
then get identified if they happen to coincide after the collapsing.

Let us describe now the fibres of the map $\theta_s$. If there are $k$
inner edges in $\theta_s(Q')$ starting at the collapsed vertex, then
there must be $k+1$ inner edges in $Q'$ starting either at $i$ or
$i+2$. There are $k+2$ choices for the distribution of these $k+1$
starting points between $i$ and $i+2$. Therefore the fibre has $k+2$
elements.

Moreover, these $k+2$ elements are naturally totally ordered by
oriented flips, as there is a unique sequence of $k+1$ oriented flips going
from the case where all $k+1$ edges start at $i$ to the case where all
$k+1$ edges start at $i+2$.  This gives a natural total order on
every fibre of $\theta_s$.

Oriented flips between $Q$-compatible quadrangulations are of two kinds.
\begin{proposition}
  \label{dichotomie}
  Let $Q' \longrightarrow Q''$ be a flip of $Q$-compatible
  quadrangulations. Then exactly one of the following holds:
  \begin{description}
  \item[(inFib)] $\theta_s(Q')=\theta_s(Q'')$ and the flip goes down one step in this fibre of $\theta_s$,
  \item[(outFib)] $\theta_s(Q') \longrightarrow \theta_s(Q'')$ is a flip between  $Q^{\setminus s}$-compatible quadrangulations.
  \end{description}
\end{proposition}
\begin{proof}
  Let $H$ the red hexagon that contains and defines the flip. Let $f$
  be the inner blue edge bounding the square $s$.

  Assume first that $f$ is a long inner edge in $H$. Then $H$ must has
  two opposite sides made of an edge starting from $i$ and an edge
  starting from $i+2$. There is at most one such hexagon where an
  out-flip is possible, and this flip does not change the image by
  $\theta_s$, but moves down by one step in the total order inside the
  fibber. This is the situation \textbf{(inFib)}.

  Otherwise, $f$ is not a long inner edge in $H$. In this case, the
  hexagon $H$ remains an hexagon in the reduced quadrangulation
  $Q^{\setminus s}$, hence the flip induces a flip inside
  $Q^{\setminus s}$. This is the situation \textbf{(outFib)}.
\end{proof}

We will say that these are \textit{in-fibre flips} and \textit{out-fibre flips}.

\begin{lemma}
  \label{lemmeAB}
  For every sequence of two flips $Q' \overset{\textbf{(inFib)}}{\longrightarrow} Q'' \overset{\textbf{(outFib)}}{\longrightarrow} Q'''$, one can find
  a quadrangulation $Q^{\sharp}$ and flips
  \begin{equation}
  Q' \overset{\textbf{(outFib)}}{\longrightarrow} Q^{\sharp} \overset{\textbf{(inFib)}}{\longrightarrow} \dots \overset{\textbf{(inFib)}}{\longrightarrow} Q'''.
\end{equation}
\end{lemma}
\begin{proof}
  If the two hexagons involved in the two consecutive flips do not
  overlap, the flips do just commute. One can simply exchange them.

  Otherwise everything happens inside an octagon. There are
  essentially just two cases to consider: there can be either $4$ or
  $5$ ways to fill the octagon. In both cases, one can find the
  expected sequence of flips by inspection. It will involve $1$ or $2$ final
  in-fibre flips.
\end{proof}

\begin{proposition}
  \label{tri_out_in}
  For every pair of quadrangulations $Q'$ to $Q''$ related by a
  sequence of flips, one can find a sequence of flips from
  $Q'$ to $Q''$ starting with out-fibre flips and ending
  with in-fibre flips.
\end{proposition}
\begin{proof}
  This follows from the previous lemma \ref{lemmeAB}, by repeated
  application at the first possible place in the sequence. This
  rewriting preserves the number of out-fibre flips. This
  process is finite, because at every reduction step, the first
  out-fibre flip not already packed at the beginning get
  strictly closer to the beginning of the sequence.
\end{proof}

Using these tools, one can now proceed to proofs by induction on $n$.

\begin{proposition}
  Every sequence of flip-moves is finite.
\end{proposition}
\begin{proof}
  By induction on the number of squares $n$. This is obvious for the
  case $n=1$, where there is no flip at all.

  Otherwise, let us consider an infinite sequence of flips. One can
  map it to a sequence made either of flips in the smaller
  quadrangulation $Q^{\setminus s}$ or a down-step in the total order
  of a fibre of $\theta_s$. Because these fibres are finite, one can
  obtain an infinite sequence of flips in $Q^{\setminus s}$. This is
  absurd by induction.
\end{proof}

It follows that there cannot be any oriented cycles in $\Gamma_Q$.

\begin{theorem}
  The flip graph $\Gamma_Q$ is the Hasse diagram of a partial order.
\end{theorem}
\begin{proof}
  The proof is by induction on $n$. One has to show that the graph is
  transitively reduced.

  This means that for a single flip $Q' \longrightarrow Q''$,
  there should be no longer sequence of flips starting at $Q'$ and ending at
  $Q''$. Here longer means having at least $2$ steps.
  
  Assume first that the flip $Q' \longrightarrow Q''$ is inside a
  fibre of $\theta_s$. Then any flip exiting this fibre would never
  come back to this fibre by induction. So it is only possible to look
  for the longer sequence inside the fibre. But there cannot be any
  longer path between two consecutive elements of a total order.

  Assume now that the flip $Q' \longrightarrow Q''$ induces a flip
  $\theta_s(Q') \longrightarrow \theta_s(Q'')$. Assume moreover that there is
  some longer sequence of flips from
  $Q' \longrightarrow Q''$. By proposition
  \ref{tri_out_in}, one can suppose that it starts with out-fibre
  flips and ends with in-fibre flips.

  If the longer sequence starts with an out-fibre flip, this must be
  the flip $Q' \longrightarrow Q''$, otherwise one cannot reach $Q''$,
  by induction hypothesis. It must then stop, otherwise it will just
  stop at some lower point in the fibre. So in fact, it was not
  longer, which is absurd.

  If the longer sequence starts with an in-fibre flip, it can only
  contain in-fibre flips, and therefore can never reach $Q''$ which is
  in a different fibre. This is a contradiction.
\end{proof}

By convention, the partial order will be decreasing along the edges of
the oriented graph $\Gamma_Q$. 

\begin{proposition}
  The map $\theta_s$ is a morphism of posets.
\end{proposition}
\begin{proof}
  This is clear by the description of flips using $\theta_s$ in proposition \ref{dichotomie}.
\end{proof}

%

Recall the definition of two special $Q$-compatible quadrangulations $Q$ and $\tau(Q)$ at the end of \autoref{compat}.

\begin{proposition}
  \label{unique_minmax}
  The unique $Q$-compatible quadrangulation with only out-flips is
  $Q$. The unique $Q$-compatible quadrangulation with only in-flips is
  $\tau(Q)$.
\end{proposition}
\begin{proof}
  Let us first note that $Q$ has only out-flips. This follows from its
  definition by a small negative rotation. Similarly, $\tau(Q)$ has
  only in-flips, because it is defined by a small positive rotation.

  It remains to show the converse statements. Let us prove that for
  $Q$ and out-flips, the case of $\tau(Q)$ and in-flips being similar.
  The proof is illustrated in \autoref{leaf_induction}.

  Let $Q'$ be a $Q$-compatible quadrangulation with only out-flips.

  Let $s=(i,j,k,\ell)$ be a square of $Q$ that is a leaf in the tree
  structure of $Q$, where the boundary edges are $i - j, j - k$ and
  $k - \ell$. One can assume without restriction that $j$ is a black
  vertex and $k$ a white vertex.

  Let $\hat{j}$, $\hat{k}, \hat{\ell}$ be the boundary edges of $Q'$
  corresponding in the obvious way to the vertices $j, k, \ell$. If
  the unique square of $Q'$ containing $\hat{j}$ and $\hat{k}$ also
  contains $\hat{\ell}$, then one can remove the boundary edges $i - j,
  j - k$ and $k - \ell$ from $Q$ and the boundary edges $\hat{j}$,
  $\hat{k}, \hat{l}$ from $Q'$ and proceed by induction. It then
  follows that $Q'$ is equal to $Q$.

  Otherwise, there are inner edges of $Q'$ starting at the common
  vertex of $\hat{k}$ and $\hat{\ell}$. One of them is bounding the
  square of $Q'$ containing $\hat{j}$ and $\hat{k}$. Let $H$ be the
  hexagon containing this edge in its interior. One can check that
  the inner edge of $H$ is an in-flip for $Q'$, which is absurd.

  \begin{figure}[ht]
    \begin{center}
      \def\svgwidth{0.30\textwidth}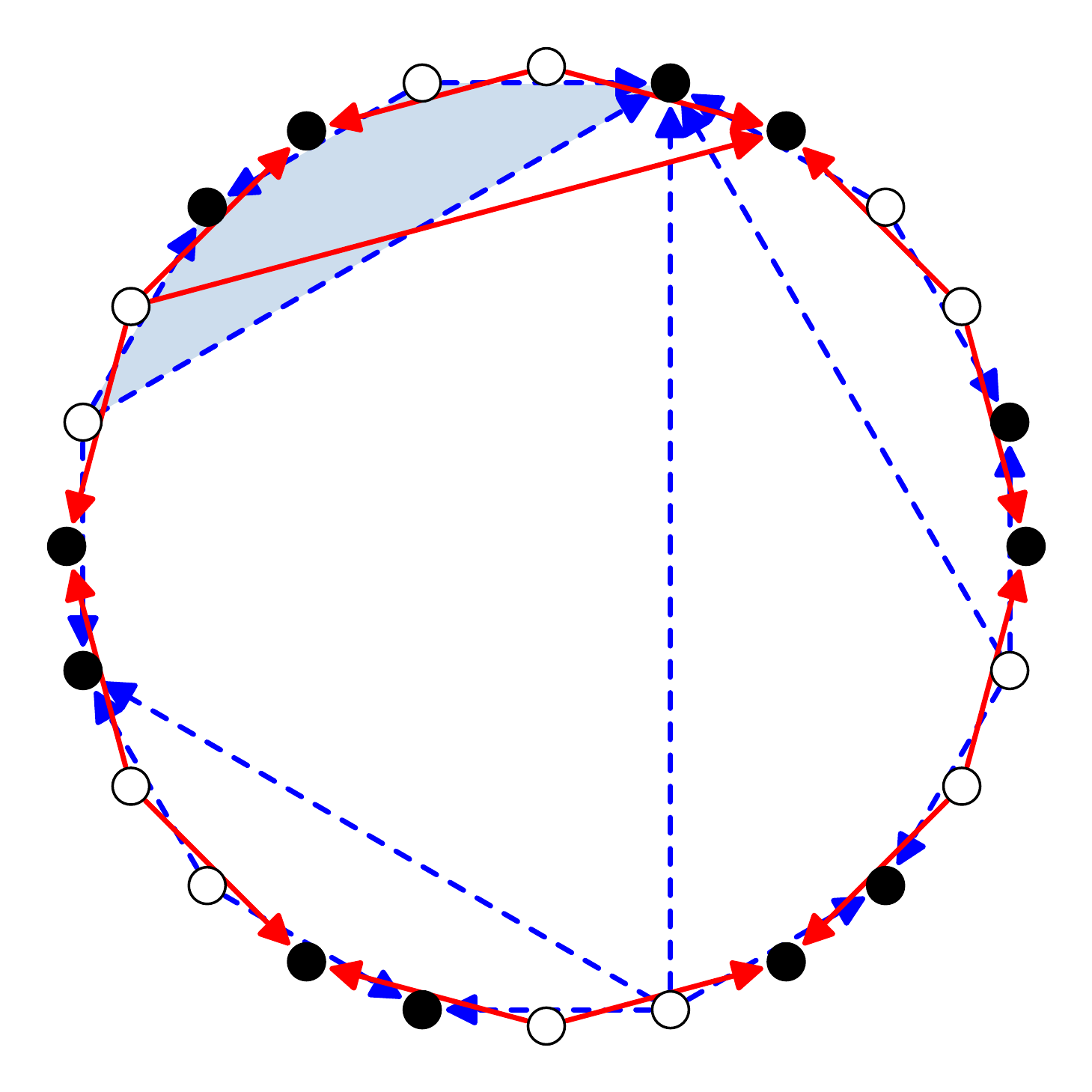
      \def\svgwidth{0.33\textwidth}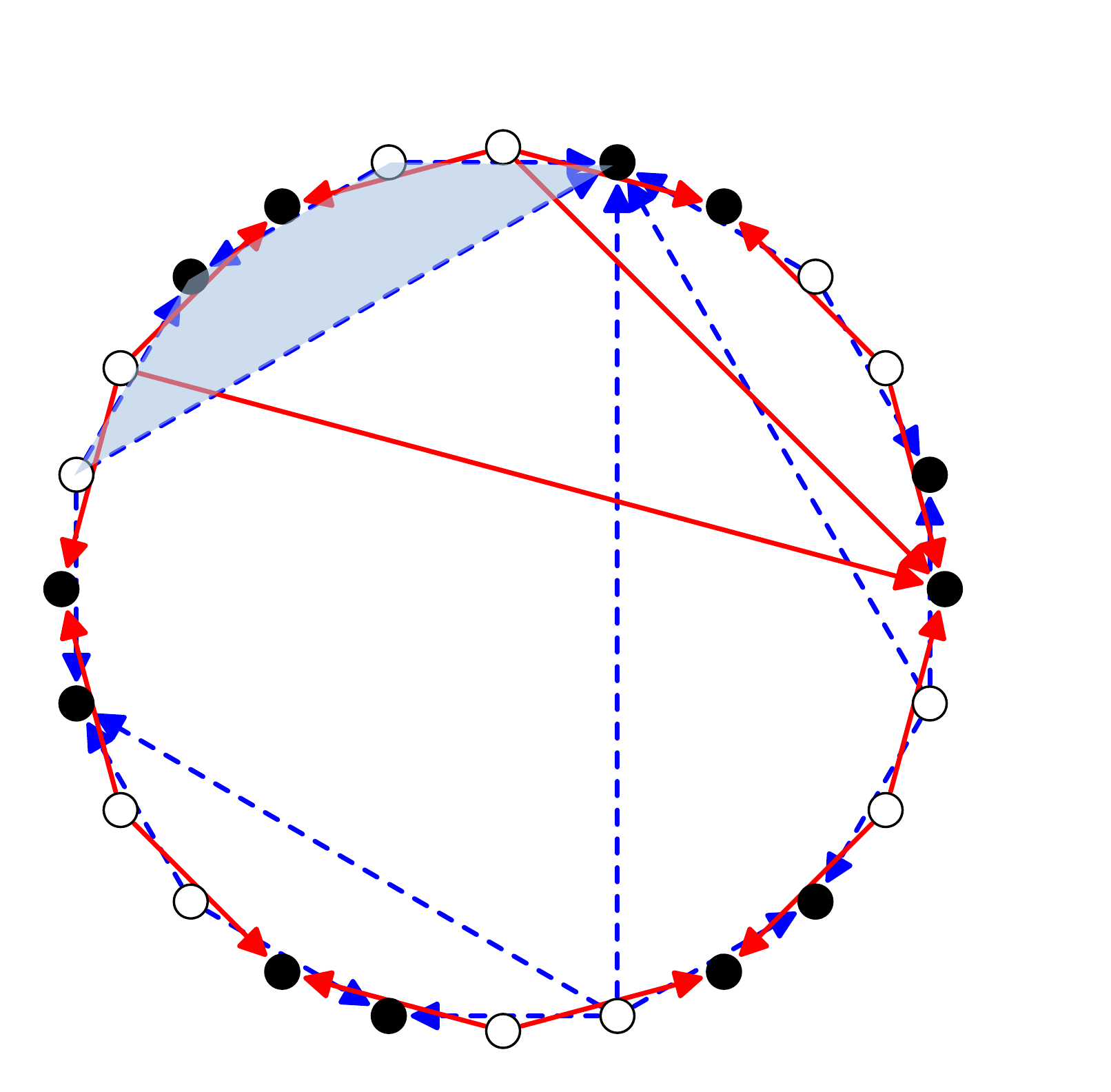
      \caption{Induction using a leaf of $Q$}
      \label{leaf_induction}
    \end{center}
  \end{figure}
\end{proof}

\begin{proposition}
  The oriented flip graph $\Gamma_Q$ is the Hasse diagram of a
  connected poset. The undirected flip graph $\gamma_Q$ is connected.
\end{proposition}
\begin{proof}
  This follows from the existence of a unique minimal element $\tau(Q)$
  (and a unique maximal element $Q$) in the poset, proved in
  proposition \ref{unique_minmax}.
\end{proof}

From now on, $\Gamma_Q$ will denote both the poset and its Hasse
diagram. Note that these posets are not graded, as one can see already
with the pentagon that one can obtain starting with some of the
quadrangulations of an octagon.

\subsection{Another description ?}

\label{no_twice}

Let us now consider another potential way to define the directed graph
$\Gamma_Q$. This is only conjectural for the moment.

Fix an initial quadrangulation $Q$. Consider the directed graph $D$
with vertices all quadrangulations, and edges defined by
counter-clockwise rotation inside hexagons (where the edge extremities
are moved to adjacent vertices). This large directed graph $D$
obviously contains $\Gamma_Q$. 

One would like to find another way to recover $\Gamma_Q$, by keeping
only some vertices and edges of $D$, without having to use the
compatibility relation.

The idea is that the set of vertices of $\Gamma_Q$ should have the
following equivalent description. Let us say that a path in $D$ is
\textit{repetition-free} if no quadrangulation occurs twice. A
quadrangulation $Q'$ is said to be \textit{reachable from $Q$} if no
repetition-free path in $D$ from $Q$ to $Q'$ contains two consecutive
flips inside the same hexagon. Then quadrangulations reachable from $Q$
should be exactly the same as elements of $\Gamma_Q$.

If this holds, then one can build the set of reachable
quadrangulations step by step, using for example the following
algorithm.

\begin{itemize}
  \item Start from the singleton set $S = \{Q\}$.
  \item Assume that some set $S$ of reachable quadrangulations have been
    found. Then either 
    \begin{enumerate}
    \item there exists a quadrangulation $Q'$ not in $S$, obtained
      from an element of $S$ by a flip, reachable, and such that all
      quadrangulations on all repetition-free paths in $D$ from $Q$ to $Q'$
      are already in $S$,
    \item or there is no such $Q'$ and the algorithm stops.
      \end{enumerate}
\end{itemize}

To see what this algorithm does inside $\Gamma_Q$, let us fix a linear
extension of $\Gamma_Q$, starting with $Q$ and ending with
$\tau_Q$. Then the algorithm would succeed by adding the elements of $\Gamma_Q$
according to this linear extension.

This would allow to define the graph $\Gamma_Q$ without using the
compatibility condition introduced by Baryshnikov, at the prize of a
somewhat greater complexity. This could be useful in more general
contexts where no analogue of compatibility is known.

This description of $\Gamma_Q$ would be summarised by the following
sentence: do not rotate twice in the same hexagon.

\section{Examples and properties of Stokes posets}

\label{poset_property_section}

\subsection{Catalan and other examples}

\label{expl_cata}

Let us first consider the special case of the quadrangulations $C_{n}$
with inner edges
\begin{equation}
  (0,3),(0,5),\dots,(0,2n-1)
\end{equation}
inside the polygon with $2n+2$ vertices. They have $n$ squares. An
example is depicted in \autoref{catalan}.
\begin{figure}[ht]
  \begin{center}
    \includegraphics[height=2cm]{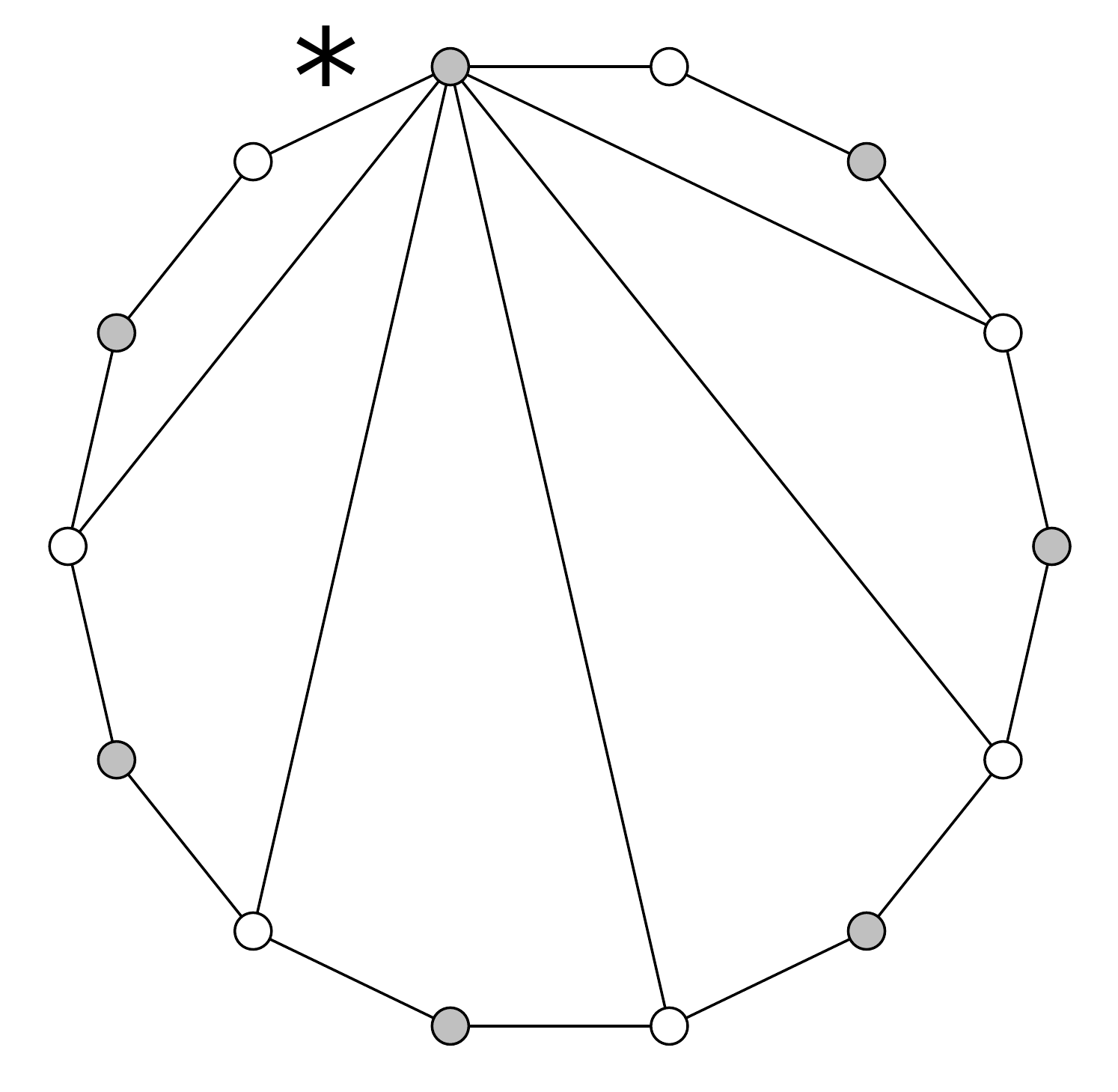}
    \caption{The quadrangulation $C_{6}$ of the Catalan family}
    \label{catalan}
  \end{center}
\end{figure}

Let us choose the border edge of the red polygon corresponding to the
vertex $0$ of the blue polygon as a root. Recall that choosing such a
root allows to identify quadrangulations with rooted ternary trees, by planar
duality. The root is marked by $\ast$ in the figure.

\begin{proposition}
  The $C_n$-compatible quadrangulations are in bijection with rooted
  ternary trees with no middle branch, and therefore also with rooted
  binary trees. The flips correspond to the left-to-right rotation
  moves of rooted binary trees. In-flips correspond to right-branches
  and out-flips to left branches.
\end{proposition}
\begin{proof}
  This description holds at the starting quadrangulation, which has
  only left branches. It remains to show that this correspondence is
  preserved under flips. Details are left to the reader.
\end{proof}

Therefore, in this case, the number of $Q$-compatible quadrangulations
is the Catalan number
\begin{equation}
  \frac{1}{n+1}\binom{2n}{n},
\end{equation}
and the poset $\Gamma_Q$ is the Tamari lattice (see for example \cite{tamfest} for its usual definition using rotation of binary trees).

More generally, for \textit{ribbon quadrangulations} (defined as
quadrangulations where no square has neighbour squares on opposite
sides), the posets $\Gamma_Q$ are expected to match with the Cambrian
lattices of type $\TA$.

\medskip

\autoref{guy_12} illustrates the directed flip graph $\Gamma_Q$ for a
quadrangulation with $4$ squares, having $12$ compatible
quadrangulations. This is the simplest example of a Stokes poset that
is really new, as the quadrangulation is not a ribbon one.
\begin{figure}[ht]
  \begin{center}
    \input{essai_2_poly.tex}\includegraphics[height=2cm]{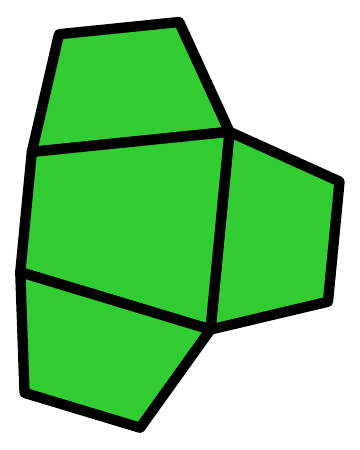}
   \end{center}
   \caption{A quadrangulation $Q$ and the flip graph of its $12$ $Q$-compatible  quadrangulations} 
  \label{guy_12}
\end{figure}

See \autoref{taille5} for the numbers of $Q$-compatible
quadrangulations for some (connected) quadrangulations with up to $6$
squares.
\begin{figure}[ht]
  \begin{center}
    \includegraphics[height=5cm]{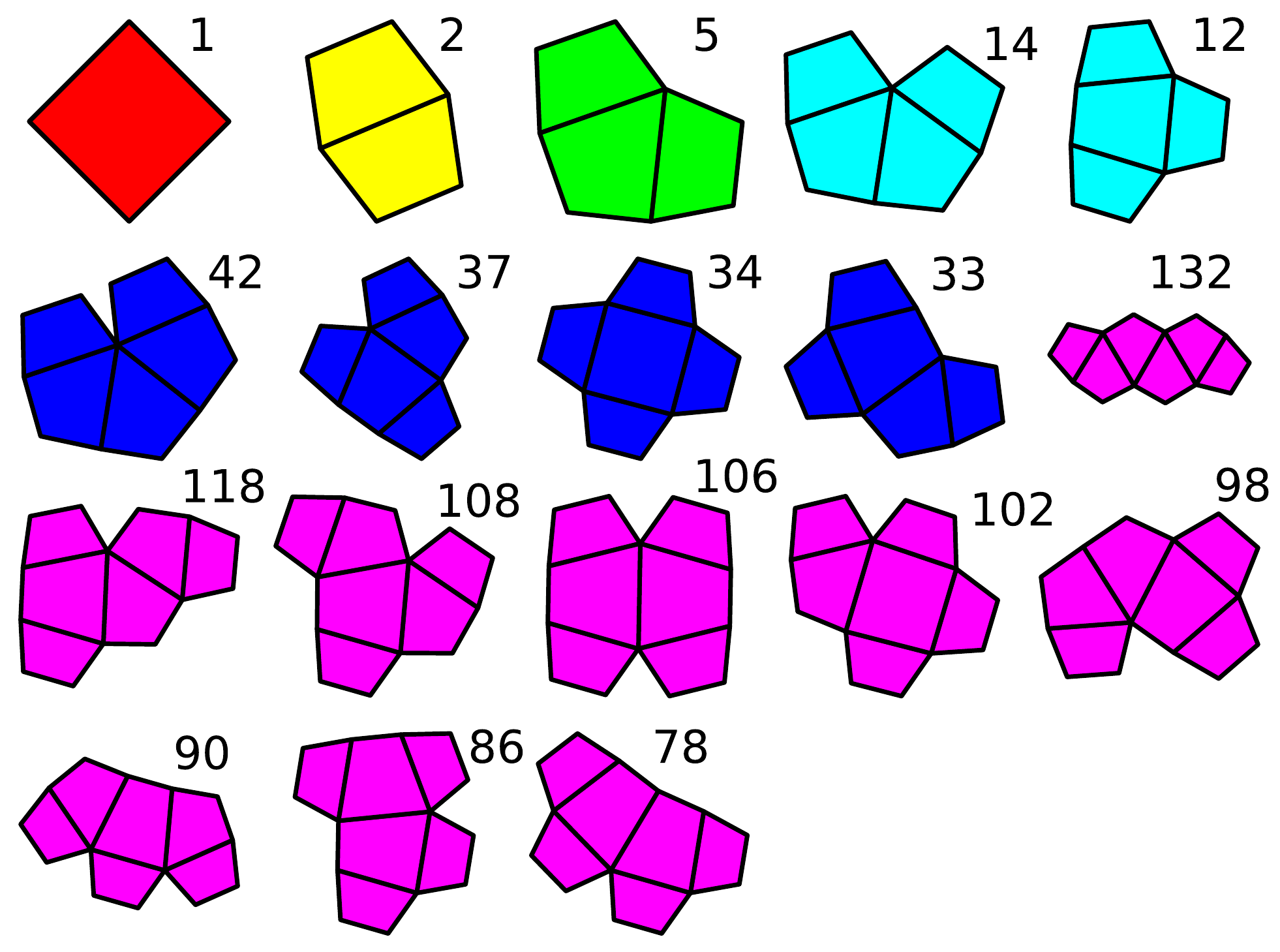}
  \end{center}
  \caption{Quadrangulations with up to $6$ squares and number of compatible quadrangulations.}
   \label{taille5}
\end{figure}

\subsection{Factorisation property}

\label{prop_bridge}

Let us now describe a factorisation property, that allows to restrict
the attention to a special class of quadrangulations.

A square $s$ in a quadrangulation is called a \textit{bridge} if it
has exactly two neighbour squares and these squares are attached to
opposite edges of $s$. A quadrangulation is called \textit{connected}
if it does not contain any bridge. See \autoref{des_ponts} for a
quadrangulation with several bridges.

\begin{figure}[ht]
  \begin{center}
    \includegraphics[height=2cm]{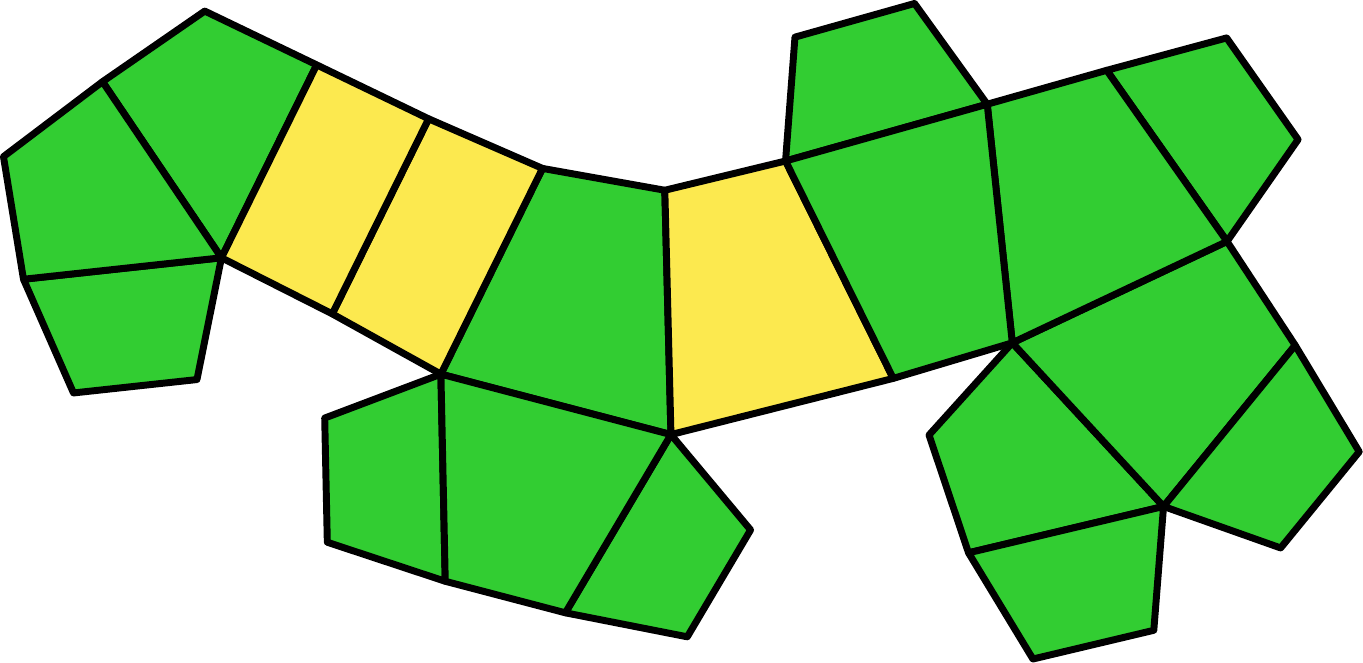}
  \end{center}
  \caption{A quadrangulation with $3$ bridges (coloured yellow).}
   \label{des_ponts}
\end{figure}

Let $Q$ be a quadrangulation. Assume that $s$ is a bridge in $Q$. Let
$Q_1$ (resp. $Q_2$) be the quadrangulation obtained from $Q$ by
removing all squares from one side of $s$ (resp. from the other
side). Equivalently, $Q_1$ and $Q_2$ are obtained by
cutting along one edge of $s$ and keeping the part that contains $s$.

\begin{proposition}
  \label{dg_prod}
  The directed graph $\Gamma_Q$ is the direct product of the directed
  graphs $\Gamma_{Q_1}$ and $\Gamma_{Q_2}$.
\end{proposition}
\begin{proof}
  The hypothesis on $Q$ says that in the polygon there are four
  vertices $i,i+1$ and $j,j+1$ that form the given square $s$ of
  $Q$. Because the edges $(i,j+1)$ and $(j,i+1)$ are oriented in
  opposite directions, no $Q$-compatible edge can cross them
  both. This implies that $Q$-compatible quadrangulations can be
  described as pairs made of one $Q_1$-compatible quadrangulation and
  one $Q_2$-compatible quadrangulation. One can check that the flips
  occur in each factor independently.
\end{proof}

\subsection{Simplicial complex and $F$-triangle}

\label{ftriangle}

Let $Q$ be a fixed quadrangulation. Let $G_Q$ be the simplicial
complex whose simplices are the sets of non-crossing $Q$-compatible
inner edges. Every $Q$-compatible quadrangulation correspond to a
maximal simplex in $G_Q$.

This simplicial complex is pure, because every partial $Q$-compatible
quadrangulation can always be completed into a $Q$-compatible
quadrangulation. To see this, one can observe that a set of
$Q$-compatible edges cuts the polygon into pieces, and each piece
behave with respect to compatibility as a smaller polygon, with an
underlying quadrangulation inherited from $Q$. 

The simplicial complex $G_Q$ is the flag complex, on the
ground set of $Q$-compatible inner edges, for the compatibility given
by being non-crossing.

Let $I^-_Q$ be the set of initial inner edges in $Q$ (seen as $Q$-compatible inner edges), and let $\Phi^+_Q$ be the
set of all other $Q$-compatible inner edges.

Let us define the $F$-triangle as the polynomial in two variables
\begin{equation}
  F_Q(x,y) = \sum_{f \in G_Q} x^{\# f \cap \Phi^+_Q} y^{\# f \cap I^-_Q}.
\end{equation}
When evaluated at $(x,x)$, this reduces to the usual $f$-vector of the
simplicial complex $G_Q$.

There is some kind of parabolic structure in quadrangulations, akin to
the classical theory of roots systems or Coxeter groups. The role of
simple roots is played by the inner edges. 
\begin{proposition}
  \label{parabo_F}
  There holds
  \begin{equation}
    y \partial_y F_Q = y \sum_{e \in Q} F_{Q^-_e}F_{Q^+_e},
  \end{equation}
  where the sum runs over inner edges of $Q$, and $Q^-_e,Q^+_e$ are
  the quadrangulations obtained by cutting $Q$ along $e$.
\end{proposition}
\begin{proof}
  The left hand side is counting simplices in the simplicial complex
  $G_Q$, with a marked initial inner edge. These are clearly in
  bijection with the disjoint union, over inner edges of $Q$, of pairs
  of simplices in the simplicial complexes $G_{Q^+_e}$ and $G_{Q^-_e}$.
\end{proof}

Let now $s$ be a bridge in a quadrangulation $Q$, as defined in \autoref{prop_bridge}.

Let $Q_1$ (resp. $Q_2$) be the quadrangulation obtained from $Q$ by
removing all squares from one side of $s$ (resp. from the other side).

\begin{proposition}
  \label{simplicial_prod}
  The simplicial complex $\Gamma_Q$ is the direct product of the
  simplicial complexes $\Gamma_{Q_1}$ and $\Gamma_{Q_2}$. The
  $F$-triangle of $Q$ is the product of the $F$-triangles of $Q_1$ and
  $Q_2$.
\end{proposition}
\begin{proof}
  This follows from the same basic properties of the compatibility
  relation as prop. \ref{dg_prod}.
\end{proof}

In the work of Baryshnikov \cite{baryshnikov}, all the simplicial
complexes $G_Q$ are described as the dual of simple polytopes. In
particular, they are all spherical, and we will use here this
result. It would certainly be interesting to study the fans and the
polytopes that are involved.

The $F$-triangle has a nice symmetry.
\begin{proposition}
  \label{F_sym}
  There holds
  \begin{equation}
    F_Q(-1-x,-1-y) = (-1)^n F_Q(x,y).  
  \end{equation}
\end{proposition}
\begin{proof}
  This is a simple consequence of the fact that $G_Q$ are spherical
  simplicial complexes, see \cite[Prop. 5]{chap_enum} for the proof of
  this equation in a very similar context.
\end{proof}

\subsection{Twisting along an edge}

\label{twisting}

Let $Q$ be a quadrangulation, and $e$ be a inner edge of $Q$. Let $Q'_e$
and $Q''_e$ be the two quadrangulations defined by cutting $Q$ along
$e$. Let $\overline{Q''_e}$ be the image of $Q''_e$ by a reflection in the
plane (the mirror image of $Q''_e$).

The \textit{twist} of $Q$ along $e$ (on the $Q''_e$ side) is the quadrangulation
defined by gluing back $Q'_e$ and $\overline{Q''_e}$ along their boundary
edges corresponding to $e$.

Twisting twice on the same side gives back $Q$. Twisting successively
on opposite sides give the mirror image of $Q$.

Let $Q$ and $Q'$ be quadrangulations related by twisting along
one inner edge.
\begin{conjecture}
  \label{twisting_ok}
  The undirected flip graphs $\gamma_Q$ and $\gamma_{Q'}$ are isomorphic. The simplicial complexes $G_{Q}$ and $G_{Q'}$ are isomorphic.
\end{conjecture}
It should be noted that the oriented flip-graphs are not the same
under this hypothesis.

In particular, the number of $Q$-compatible quadrangulations should
be the same as the number of $Q'$-compatibles ones.

The statement of conjecture \ref{twisting_ok} was already proposed in
\cite{baryshnikov} for the dual polytopes.

Moreover, one also expects some stronger enumerative invariance.
\begin{conjecture}
  \label{twisting_F}
  The $F$-triangles of $G_{Q}$ and $G_{Q'}$ are equal.
\end{conjecture}

\section{Noncrossing trees and exceptional sequences}

\label{exc_section}

In this section, which does not contain much details, some
generalisations of the flip graphs are proposed, using exceptional
sequences on Dynkin diagrams or equivalent objects for Coxeter
groups. Beware that the correctness of this proposal depends on the
unproven alternative description of $\Gamma_Q$ in \autoref{no_twice}.

Let us start by a simple combinatorial reformulation of the
flip graphs of quadrangulations in terms of other combinatorial
objects.

A \textit{noncrossing tree} in the regular polygon with $n+1$ vertices is a set
of edges between vertices of the polygon, with the following properties
\begin{itemize}
  \item edges do not cross pairwise,
  \item any two vertices are connected by a sequence of edges,
  \item there is no loop made of edges. 
\end{itemize}
Note that boundary edges are allowed in the set, and a typical
noncrossing tree will contain only some of them.

There is a simple bijection between quadrangulations of the
$2n+2$-polygon and noncrossing trees in the $n+1$-polygon. Assume that
vertices of the $2n+2$-polygon have been coloured black and white
alternating. Then every quadrangle contains a unique black-black
diagonal (one of its two diagonals). The collection of these diagonal edges can be seen to form a
noncrossing tree. Conversely, by drawing a noncrossing tree using the
black vertices of the $2n+2$-polygon, one can consider all black-white
edges that do not cross edges of the noncrossing tree. This gives
back the quadrangulation.

\begin{figure}[ht]
  \begin{center}
    \includegraphics[height=3cm]{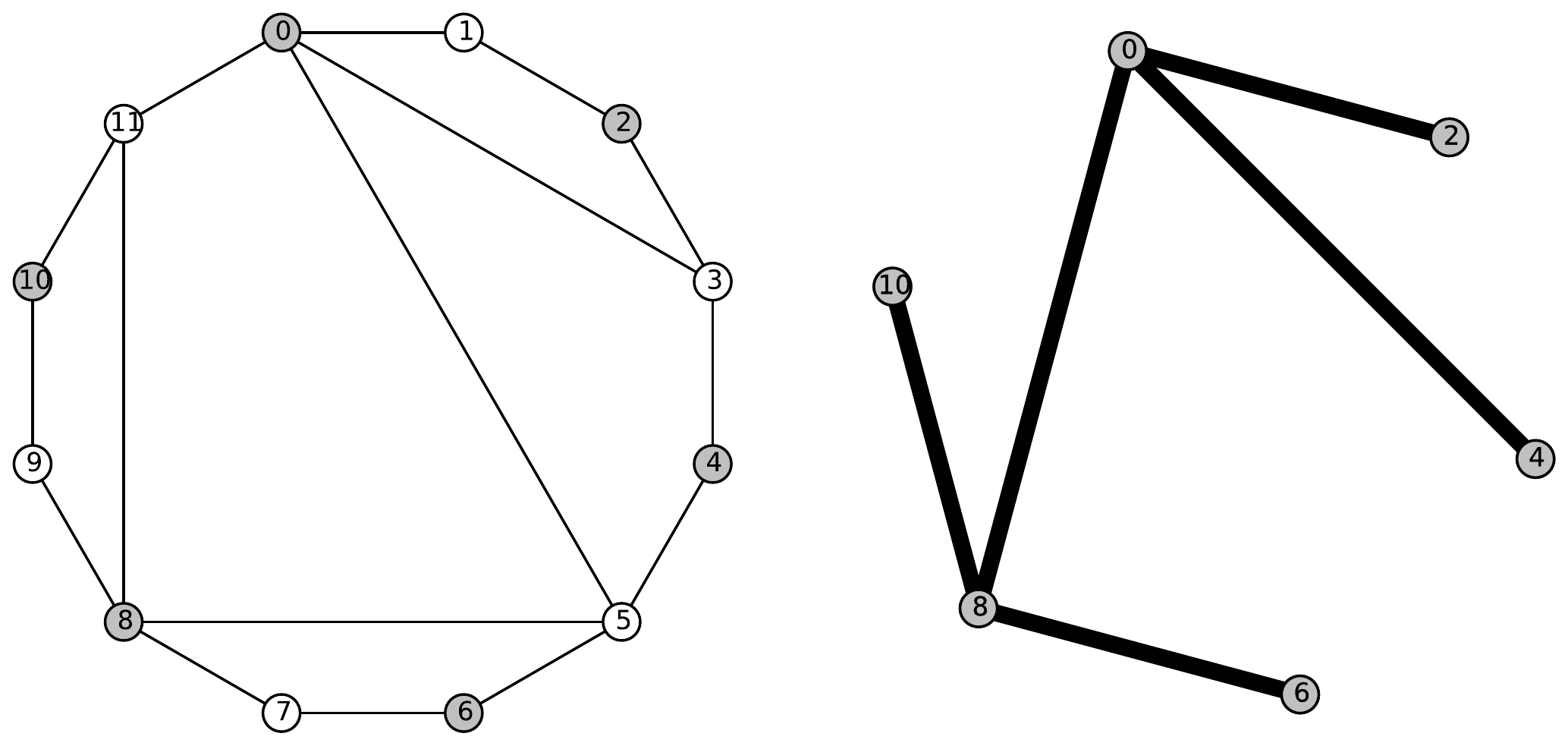}
  \end{center}
  \caption{Bijection between quadrangulations and noncrossing trees}
   \label{bij_quad_nct}
\end{figure}

Passing through this bijection, one can translate the flips of
quadrangulations into a simple operation on noncrossing trees. It is
in fact enough to consider what happens for the flip inside an hexagon,
as every flip will behave locally the same.

The result of as follows. A flip of noncrossing trees will change just
one edge of the noncrossing tree. Assume that there are edges $i - j$
and $k - j$ in the noncrossing tree $T$, with $i<j<k$ in the
clockwise cyclic order, and that there is no other edge incident to
$j$ between them inside the ambient polygon. Then the flip consists in
replacing $i - j$ by $i - k$.

Flipping twice in the same hexagon of a quadrangulation translates
into flipping again along two other sides of the same triangle
$(i,j,k)$ of a noncrossing tree.

Using the conjectural alternative description of \autoref{no_twice},
one would therefore be able to restate the digraph $\Gamma_Q$ in terms
of flips of noncrossing trees, not flipping twice consecutively along
two sides of the same triangle.

\subsection{From noncrossing trees to exceptional sequences}

For short, let $\TA_n$ denote in this section the equi-oriented quiver of type $\TA_n$. The
noncrossing trees in the $n+1$-polygon have been related in
\cite{araya} to exceptional sequences (up to shifts and reordering) in
the derived category $D\mod \TA_{n}$.

The correspondence goes as follows. By numbering the vertices of the
$n+1$-polygon from $0$ to $n$ clockwise, every inner edge gets a label
$(i,j)$ with $i<j$. Let the inner edge $(i,j)$ correspond to the
indecomposable module over $\TA_n$ with support $[i+1,j]$.

Then Araya proved that a noncrossing tree is sent by this map to a
collection of indecomposable modules that can be ordered into an exceptional sequence
and, conversely all the underlying sets of indecomposable modules coming from
exceptional sequences are obtained in this way. Let us call such a
collection of indecomposable modules an \textit{exceptional set}\footnote{It could also be called an exceptional collection, but there is no universal agreement on the meaning of that.}.

Given the description above of the flips acting on noncrossing trees,
one can readily check that they correspond, after the bijection to
exceptional sets, to the action of the braid group on exceptional
sequences (see for example \cite{crawley} for the braid group
action). More precisely, given an exceptional set, pick an ordering
into an exceptional sequence, act by one generator $s_i$ of the braid
group to get another exceptional sequence, and then forget about the
ordering. This describe what the flips are. 

Then moving twice consecutively in a same triangle gets translated to
acting twice by the same $s_i$ at the same place.

This would allow to generalise the definition of the flip graphs to
any finite Dynkin diagram, provided that the conjectural description
of \autoref{no_twice} holds.

\subsection{From exceptional sequences to chains of noncrossing partitions}

In fact, it is even possible to get from here to the more general
setting of finite Coxeter groups.

It has been known since the article \cite{ingalls_thomas} (see also
\cite{igusa_schiffler}) that exceptional sequences in derived
categories of representations of Dynkin quivers are closely related to
the corresponding lattices of noncrossing partitions.

More precisely, let $W$ be a finite crystallographic Coxeter group,
and $c$ be a Coxeter element in $W$. As is well-known, this data can
be translated into a quiver $Q$ with underlying graph the Dynkin
diagram.

There is a bijection between maximal chains in the noncrossing
partition lattice attached to $c$, and exceptional sequences of
modules over $Q$. In fact, maximal chains in the noncrossing partition
lattices are the same as factorisations of the Coxeter element $c$ as a
product of reflections\footnote{All reflections, not only simple ones.}.

In this new language, the action of the braid group on exceptional
sequences get translated to a similar action of the
braid group on maximal chains of the noncrossing partition
lattice\footnote{Sometimes named the Hurwitz action.}. When thinking to maximal chains as factorisations of the
Coxeter element into reflections, this is given by a simple conjugation action.

This interpretation would allow to extend the definition of the flip graphs
to all finite Coxeter groups, where quadrangulations are replaced by
maximal chains in the noncrossing partition lattice, up to reordering.

In these generalisations, some properties are lost. In particular, the
flip graphs are no longer regular, as one can see already for some
examples in type $\TD_4$.

\section{Serpent nests}

\label{nests_section}

In this section, we turn to another aspect of quadrangulations, that
should in fact be related to the previous one.

Let us start by some combinatorial definitions.

A \textit{serpent}\footnote{The word ``serpent'' is a synonym for
  ``snake'', which is already used for something else in the context
  of cluster algebras.} in a quadrangulation $Q$ is a set of two distinct
squares $s,s'$ such that the unique path from $s$ to $s'$ in $Q$ is
made of a sequence of right-angle turns. This means that following
this path, one never enters and leaves a square through opposite
sides. Abusing notation, one will also consider that the path is the
serpent. By convention, the path starts and ends at the centre of the
squares. See the leftmost part of \autoref{serpent_nest} for a visual example.

A \textit{serpent nest} is a set of serpents inside a quadrangulation,
satisfying an extra condition and modulo an equivalence relation:
\begin{itemize}
\item The \textbf{extra condition} is the following: two serpent extremities $s_1$
and $s_2$ can share the same square if and only if the serpents do not
leave this square by the same side.
\item The \textbf{equivalence relation} is the following: two sets of
serpents are considered equivalent if the pattern of paths inside every
square is the same in both.
\end{itemize}

Let us explain in more details the meaning of this equivalence
relation. Inside every square, there can be 8 kinds of path segments:
4 kinds entering by an edge and stopping at the centre (these can
appear at most once), and 4 kinds entering by an edge and leaving by
an adjacent edge (these can appear any number of times). The pattern
of paths just remembers how many paths segments of each kind there are.
\begin{figure}[ht]
  \begin{center}
    \includegraphics[height=2cm]{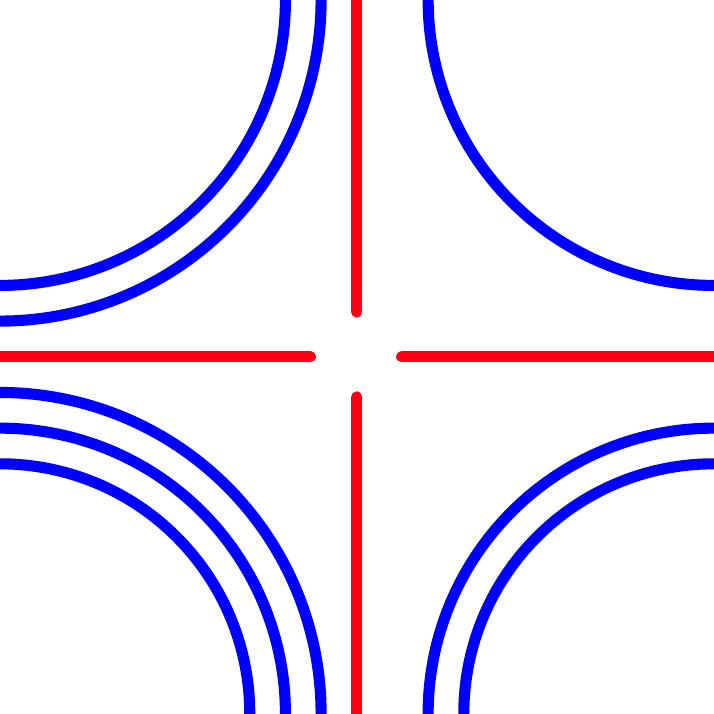}
  \end{center}
  \caption{An example of the pattern of paths in a square} 
\end{figure}

Another way to describe this equivalence relation is the
following. Consider two serpents that cross the same edge. Cut both of
them into two pieces along this edge, and glue them back after
swapping them. This gives another set of serpents that also satisfies the extra
condition. The equivalence relation is the closure of this kind of
move.

\begin{figure}[ht]
  \begin{center}
     \includegraphics[height=4cm]{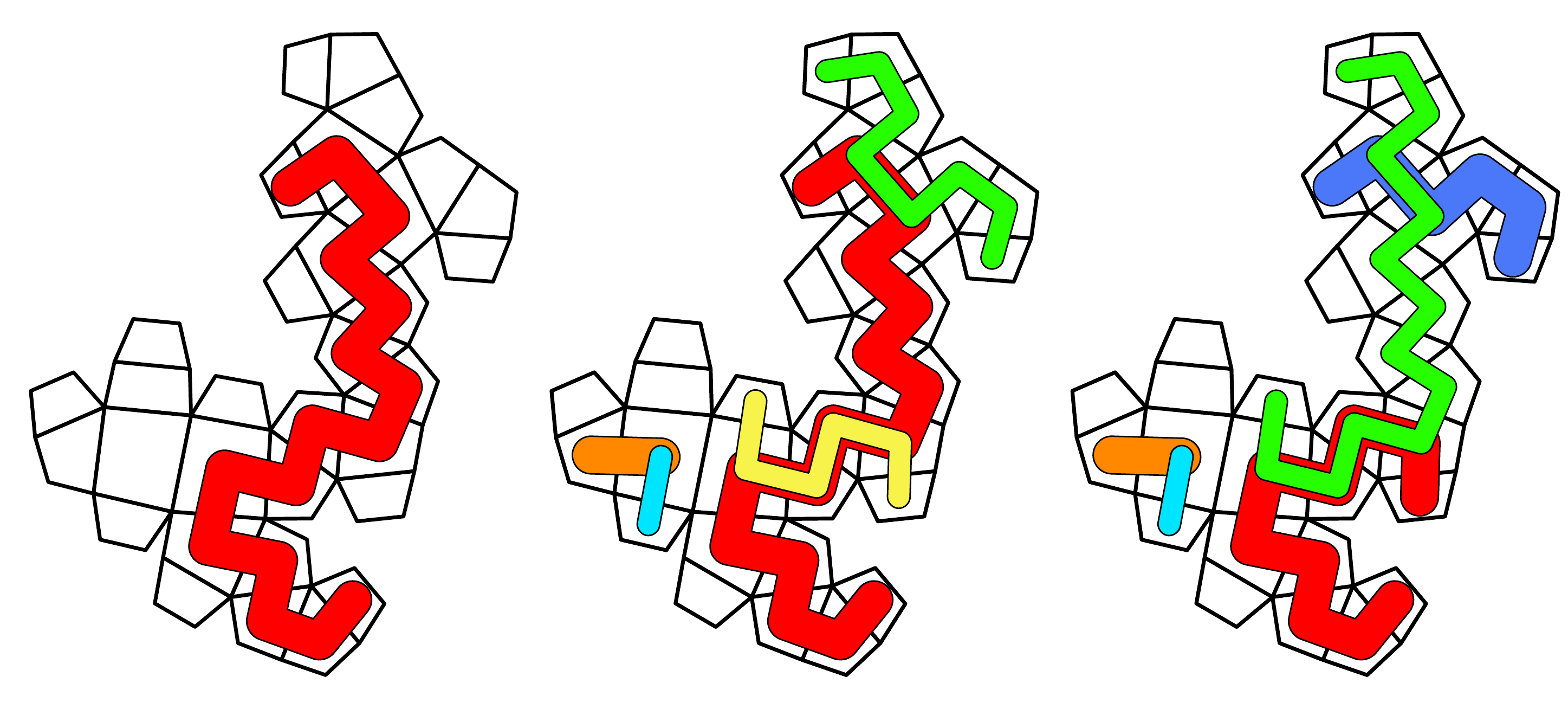}
   \end{center}
   \caption{A serpent and two representatives of the same serpent nest of rank $5$} 
   \label{serpent_nest}
 \end{figure}

The extra condition implies that every square can contain at most $4$
serpent extremities. Therefore, there are only a finite number of
different serpent nests in each quadrangulation.

Let $\SN_Q$ be the set of serpent nests in a quadrangulation $Q$.

Let us define the \textit{rank} $\rk(\sigma)$ of a serpent nest
$\sigma$ as the number of serpents in it. By a well-known principle,
this is half the number of serpent extremities.

Let $Q$ be a quadrangulation. Assume that $s$ is a bridge in $Q$ (as
defined in \autoref{prop_bridge}). Let $Q_1$ (resp. $Q_2$) be the
quadrangulation obtained from $Q$ by removing all squares from one
side of $s$ (resp. from the other side).

\begin{theorem}
  There is a bijection between serpent nests in $Q$ and pairs of
  serpent nests in $Q_1$ and $Q_2$.
\end{theorem}
\begin{proof}
  In fact, no serpent can cross the bridge, by definition, because
  they must make right-angle turns at every step. Therefore every
  serpent in $Q$ is either in $Q_1$ or in $Q_2$. The extra condition
  and equivalence relation are also compatible with this decomposition.
\end{proof}

Let now $Q$ and $Q'$ be quadrangulations related by twisting along
some inner edge $e$, as defined in \ref{twisting}. Every serpent in $Q$ can
be mapped to a serpent in $Q'$ by just twisting it. This is clearly a
bijection.

\begin{theorem}
  \label{twisting_nests}
  This bijection induces a bijection between serpent nests in $Q$ and
  in $Q'$, that preserves the rank.
\end{theorem}
\begin{proof}
  The extra condition in the definition of serpent nests is obviously
  preserved. Compatibility with the equivalence relation is easy.
\end{proof}

For a quadrangulation $Q$, let $h_Q$ be the generating polynomial for
the rank of serpent nests:
\begin{equation}
  h_Q(x) = \sum_{\sigma \in \SN_Q} x^{\rk(\sigma)}.
\end{equation}
Abusing notation, this is called the $h$-vector.

\medskip

Let us now define a duality on serpent nests.

For this we need the following definition. A \textit{short serpent} is
a serpent whose extremities are adjacent squares. We say that a
serpent nest $\sigma$ contains a short serpent at the edge $e$ if one
set of serpents in the equivalence class of $\sigma$ contains a short
serpent at this edge. Equivalently, there are serpent extremities in
each square adjacent to $e$ that both leave their respective squares
through $e$.

Let $\sigma$ be a serpent nest. The dual $\overline{\sigma}$ is
another serpent nest, defined as follows.

Consider an edge $e$ between two squares $s$ and $s'$ of $Q$. If both
$s$ and $s'$ contains a serpent extremity leaving through $e$, remove
them both. If neither $s$ nor $s'$ contains a serpent extremity
leaving through $e$, add both. Otherwise (exactly one of $s$ or
$s'$ has a serpent extremity leaving through $e$), do nothing.

This change has to be performed for every inner edge of $Q$.

This is clearly an involution on serpent nests. This amounts to remove
all existing short serpents, and introduce some wherever possible.

\begin{proposition}
  The rank of the dual of a serpent nest $\sigma$ is $n - 1 - \rk(\sigma)$,
  where $n - 1$ is the number of inner edges of the quadrangulation.
\end{proposition}
\begin{proof}
  The quadrangulation has $n$ squares and $n-1$ inner edges.

  Consider a serpent nest $\sigma$ containing $k$ serpents among which
  $\ell$ short serpents. There are $k-\ell$ long serpents, and
  therefore $2k-2\ell$ inner edges where the duality does nothing
  (near the extremities of the long serpents). The duality changes
  something around $n-1-2k+2\ell$ inner edges. So in the dual
  $\overline{\sigma}$, there are $n-1-2k+\ell$ short
  serpents. Therefore $\overline{\sigma}$ contains $n-1-k$ serpents.
\end{proof}

This implies that the $h$-vector is palindromic:
\begin{equation}
  \label{palindromic_h}
  h_Q(x) = x^{n-1} h_Q(1/x).
\end{equation}
For the four examples of size $5$ in \autoref{taille5}, one finds the following
$h$-vectors
\begin{align}
  x^{4} + 8 x^{3} + 15 x^{2} + 8 x + 1,\\
  x^{4} + 8 x^{3} + 16 x^{2} + 8 x + 1,\\
  x^{4} + 9 x^{3} + 17 x^{2} + 9 x + 1,\\
  x^{4} + 10 x^{3} + 20 x^{2} + 10 x + 1.
\end{align}

From the description of the duality, fixed points are serpent nests
where each inner edge is adjacent to exactly one serpent extremity
that crosses this edge.

It seems that the number of fixed points is (up to sign) the value of
the $h$-vector at $x=-1$. This could be related to the Charney-Davis
conjecture about $h$-vectors of simplicial spheres, see
\cite{reiner_welker, charney_davis,leung_reiner}.

\medskip

Let us now define a refinement of the $h$-vector.

A serpent in a serpent nest $\sigma$ is called a \textit{simple
  serpent} if two conditions hold. The first condition is that its
extremities are adjacent squares. The second condition is that there
is no other serpent crossing the same edge. This is a stronger
restriction than just being a short serpent.

The $h$-triangle of a quadrangulation $Q$ is defined as the
following polynomial in two variables:
\begin{equation}
  H_Q(x,y) = \sum_{\sigma \in \SN_Q} x^{\rk(\sigma)}y^{s(\sigma)},
\end{equation}
where $s(\sigma)$ is the number of simple serpents in the serpent nest $\sigma$.

Given an inner edge $e$ in a quadrangulation $Q$, one can define two
quadrangulations by cutting $Q$ along this edge. Let us call them
$Q'_{e}$ and $Q''_{e}$. This will be called a \textit{parabolic decomposition}.

The $H$-triangle has a simple property with respect to parabolic decomposition.
\begin{proposition}
  There holds
  \begin{equation}
    y \partial_y H_{Q} = x y \sum_{e \in Q} H_{Q'_e} H_{Q''_e}.
  \end{equation}
\end{proposition}
\begin{proof}
  The left hand side is counting pairs made of a serpent nest in $Q$
  with a marked simple serpent. Each term of the right hand side is
  counting pairs of serpent nests in $Q'_{e}$ and $Q''_{e}$. The
  equality is given by the bijection that cuts along the unique edge
  crossed by the simple serpent.
\end{proof}

There seems to be a close relation with the $F$-triangle defined in
\autoref{ftriangle}. Let $n$ be the number of inner edges.
\begin{conjecture}
  There holds the equation
  \begin{equation}
    \label{f_to_h}
    H_Q(x,y) = (x-1)^n F_Q\left(\frac{1}{x-1},\frac{1+(y-1)x}{x-1}\right).
  \end{equation}
\end{conjecture}
This would imply in particular a relation between the $f$-vector and
the $h$-vector, that is exactly the usual relation defining the
$h$-vector from the $f$-vector for simplicial complexes.

Given prop. \ref{F_sym}, this would imply the following symmetry.
\begin{conjecture}
  There holds
  \begin{equation}
    H_Q(x,y) = x^n H_Q(1/x,1-x+ x y).
  \end{equation}
\end{conjecture}

The very same relationship \eqref{f_to_h} between $F$-triangles and
$H$-triangles has first been conjectured in \cite{chap_enum} in the
context of Catalan combinatorics of finite Weyl groups and Coxeter
groups. In this case, the $F$-triangle is related to cluster algebras
and the $H$-triangle to nonnesting partitions. A similar formula was
conjectured in \cite{chap_autre} to relate the $H$-triangle to a third
triangle, called the $M$-triangle, defined using the Möbius numbers of
the lattices of noncrossing partitions.

These conjectures have later been generalised in
\cite[Ch. 5]{armstrong_memoir} to the Fuss-Catalan combinatorics of
Coxeter groups, where one more parameter $m$ appears. The interested
reader can find a lot of material on this subject in this reference.

This story is now completely proved, but maybe not fully understood, after
the works of several authors, see \cite{ath07, kra06a, kra06b, tza08} and
\cite{thiel_long}.

In fact, all the statements proved about these triangles in the
Catalan setting seems to extend to our current context, where the
initial data is a quadrangulation instead of a Coxeter type. The
$M$-triangle is missing, as there is no known analogue of the
noncrossing partitions here. Note that the symmetry above gets a nice
explication in the Catalan world by the self-duality of the lattices
of noncrossing partitions.

\subsection{Examples}

Recall the quadrangulations $C_{n}$ introduced in \autoref{expl_cata}, with
inner edges
\begin{equation}
  (0,3),(0,5),\dots,(0,2n-1)
\end{equation}
inside the polygon with $2n+2$ vertices, see \autoref{catalan}.

In this case, every pair of squares define a serpent and every serpent
nest is uniquely determined by the set of serpent extremities. In one
square, there can be at most two serpent extremities.

Recall that a nonnesting partition of the set $\{1,\dots,n-1\}$ is a set
of segments $(i,j)$ with $1\leq i \leq j \leq n$, such that no segment
contains another one.

By looking at serpent extremities as opening or closing parentheses,
one can find a simple bijection between serpent nests and nonnesting
partitions of $\{1,\dots,n-1\}$, that maps the number of serpents to
the number of segments.

It follows that the number of serpent nests inside $C_n$ is the
Catalan number and the $h$-vector of $C_n$ is given by the Narayana numbers.

\medskip

Another nice family of examples is given by the quadrangulations of
the general shape displayed in \autoref{lucas}, namely a linear
string of $n + 1$ squares, with one square attached below each of them
excepted the two extremes. Let us denote by $L_n$ this
quadrangulation, which has $2n$ squares. We call them the Lucas
family, for reasons explained below. Let us also denote by $K_n$ the
quadrangulation obtained from $L_n$ by removing the rightmost square
in the linear string of $n+1$ squares.
\begin{figure}[ht]
  \begin{center}
    \includegraphics[height=2cm]{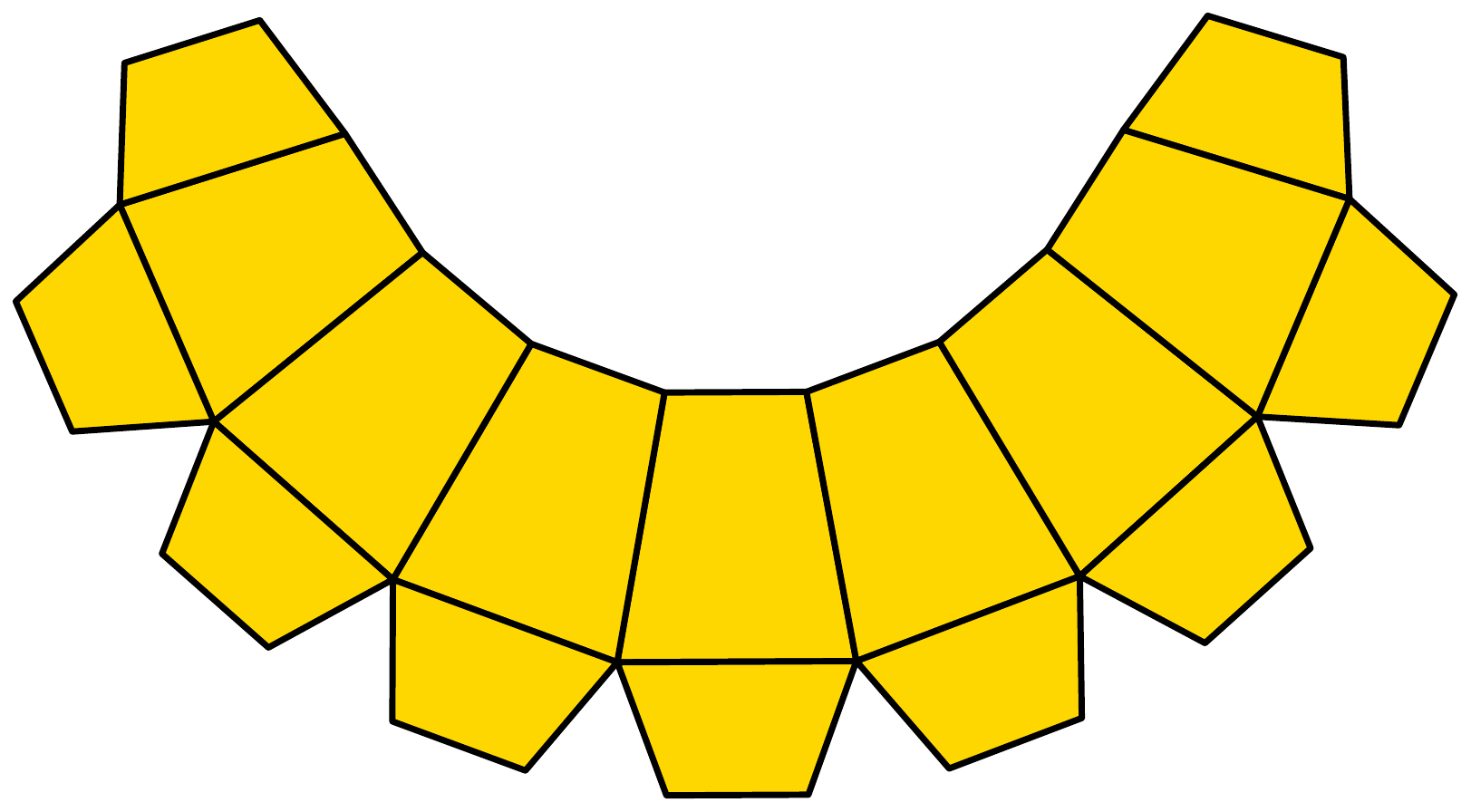}\quad\includegraphics[height=2cm]{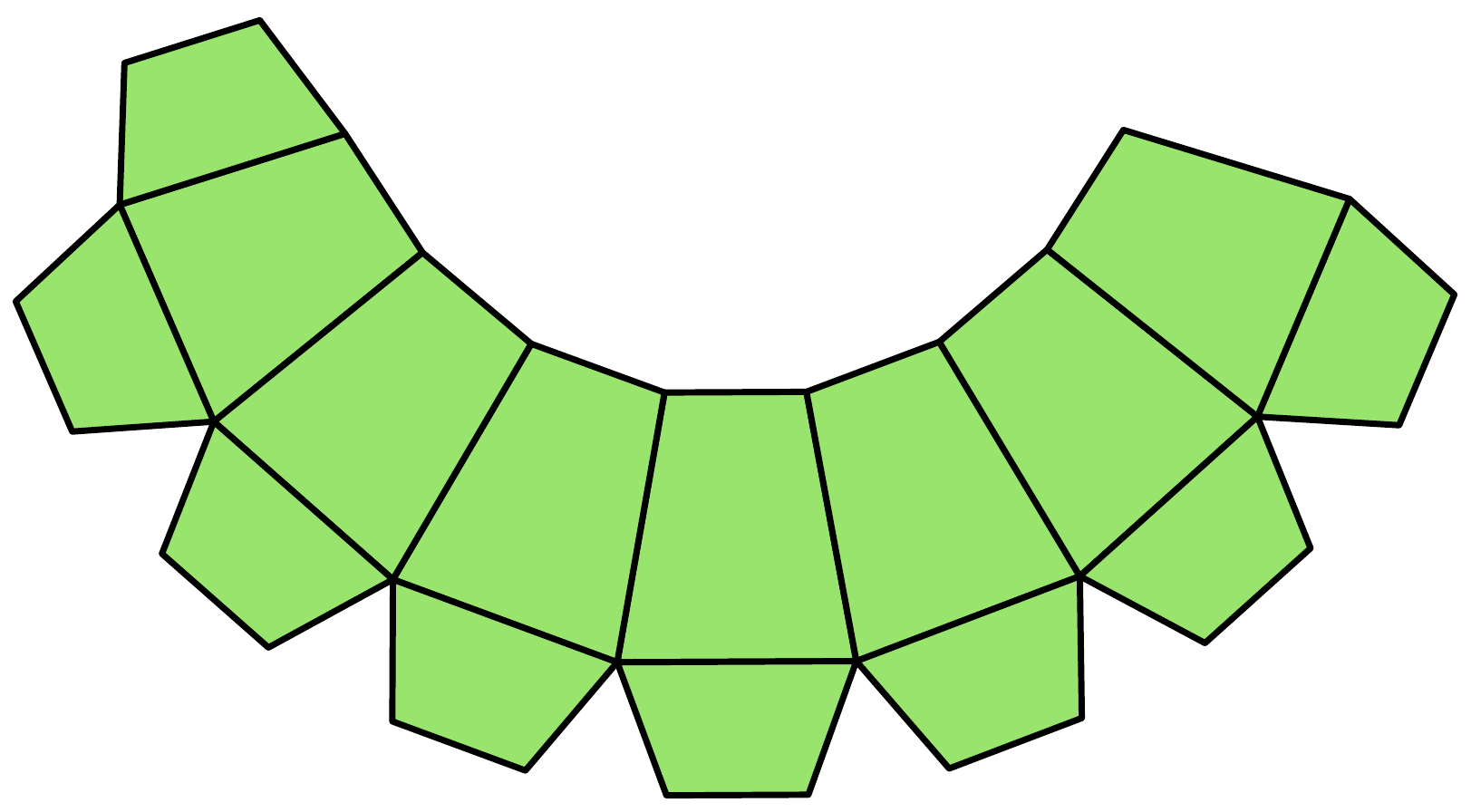}
    \caption{The quadrangulation $L_{8}$ of the Lucas family and the quadrangulation $K_8$}
    \label{lucas}
  \end{center}
\end{figure}

Let us denote by $\ell_n$ and $k_n$ the $h$-vectors $h_{L_n}$ and
$h_{K_n}$. Recall that they are polynomials in $x$ counting the
serpent nests according to their rank.

\begin{proposition}
  The sequence $\ell_n$ satisfies the recurrence
  \begin{equation}
   \ell_{n+2} = (1 + 4 x + x^2) \ell_{n+1} + x(1+x+x^2) \ell_n,
  \end{equation}
  with initial conditions $\ell_0= 0$ and $\ell_1=1+x$.
\end{proposition}
\begin{proof}
  One can obtain the following coupled recursions for $\ell_n$ and
  $k_n$ by a combinatorial reasoning on serpent nests:
  \begin{equation}
    \ell_n = (1+x) k_n + x \ell_{n-1} \quad \text{and}\quad 
    k_n = (1+x) \ell_{n-1} + x k_{n-1} + x (1 + x) \ell_{n-2}.
  \end{equation}
  By elimination, one finds the required formula.
\end{proof}
With more care, a similar recursion can be found for the $H$-triangles.

It follows that the $\ell_n$ form a Lucas sequence, and therefore that there exist polynomials $(Z_n)_{n \geq 1}$ such that
\begin{equation}
  \ell_n = \prod_{d | n} Z_d,
\end{equation}
for all $n \geq 1$. Apparently, the $Z_n$ have positive
coefficients. There is no combinatorial explanation of this
factorisation so far. One may speculate that this could be explained
by some generalisations of noncrossing partitions lattices.

The first few numbers of serpent nests in $L_n$ are
\begin{equation}
  2, 12, 78, 504, 3258, 21060, 136134, 879984, 5688306,\dots
\end{equation}
(twice sequence \href{https://oeis.org/A090018}{A090018} of \verb?oeis.org?) and the first few values of $Z_n$ at $x=1$ are
\begin{equation}
  2, 6, 39, 42, 1629, 45, 68067, 1746, 72927,
 1881, 118841823, 1737, 4965759189, \dots
\end{equation}

\subsection{Counting serpent nests}

A natural question to ask is, given a quadrangulation, how many
serpent nests does it afford ? Even better, one would like to know the
$h$-vector that counts serpent nests according to their rank.

This does not seem to be easy. There is nevertheless an idea that
allows to compute these numbers or polynomials for some large
quadrangulations. It consists of 
\begin{enumerate}
\item cutting a quadrangulation along one edge,
\item counting what could be called open serpent nests in both half-quadrangulations
\item  and gluing the results back.
\end{enumerate}

Let us call \textit{open quadrangulation} a quadrangulation together
with a distinguished boundary edge, that will be considered to be
open. An \textit{open serpent nest} inside an open quadrangulation is
defined in the same way, except that now some of the serpents can end
on the open edge. They are considered up to the same equivalence
relation as serpent nests. Clearly, there is a finite number of such
configurations.

The rank $\rk(\sigma)$ of an open serpent nest $\sigma$ is the number
of (closed) serpents. Let $\op(\sigma)$ be the number of open serpents
in $\sigma$, namely those that end on the open edge.

Given an open quadrangulation $Q$, define its \textit{open $h$-vector} as
\begin{equation}
  h_Q(x,t) = \sum_{\sigma \in \SN_Q} x^{\rk(\sigma)} t^{\op(\sigma)}.
\end{equation}

One will glue back two open $h$-vectors to get an $h$-vector. The
gluing is defined by the following rule. Let $A$ and $B$ be two
polynomials in $x$ and $t$. Define a polynomial $A\#B$ in $x$ by
\begin{equation}
  A \# B = \sum_{k\geq 0} A_k B_k x^k,
\end{equation}
where $A_k$ and $B_k$ are the coefficients of $t^k$ in $A$ and $B$
(when considered as polynomials in $t$ with coefficients in $\ZZ[x]$).

Let $Q$ be a quadrangulation, $e$ be an inner edge of $Q$. Let $Q'_e$
and $Q''_e$ be the open quadrangulations obtained by cutting $Q$ along
the edge $e$.
\begin{proposition}
  The $h$-vector of $Q$ is given by $ h_{Q} = h_{Q'_e} \# h_{Q''_e}$.
\end{proposition}
\begin{proof}
  Cutting along $e$ gives a bijection between serpent nests in $Q$ and
  pairs of open serpent nests in $Q'_e$ and $Q''_e$ with the same number of
  open serpents. This implies the formula.
\end{proof}

To give just a very simple example, consider the quadrangulation of an
hexagon. Its $h$-vector is $1+x$. On the other hand, for the open
quadrangulation of a square, one find that the open $h$-vector is
$1+t$. One can recover $1+x$ as $(1+t) \# (1+t)$.

This method allows for example to compute the number of serpents nests
for quadrangulations made by gluing a large open Catalan-type
quadrangulation (or a large open Lucas-type one) to any small open
quadrangulation.

\section{Parabolic algebras of quadrangulations and cross-trees}

\label{parabo_section}

Consider the free commutative algebra $\mathcal{F}$ generated by all
quadrangulations (considered up to rotation). Let $\mathcal{A}$ be its quotient
by the relations
\begin{equation}
  Q = Q_1 Q_2
\end{equation}
when there is a bridge $s$ in a quadrangulation $Q$ defining two
quadrangulations $Q_1$ and $Q_2$ as in \autoref{prop_bridge}.

Then $\mathcal{A}$ is a graded commutative algebra, where the degree
is the number of inner edges.

On this algebra, let us define an operator $\partial : \mathcal{A}
\to \mathcal{A}$. It is given on a quadrangulation $Q$ by
\begin{equation}
  \partial Q = \sum_{e \in Q} {Q'_e Q''_e}
\end{equation}
where the right hand side is the product of the two pieces obtained by
cutting along the edge $e$. This extends uniquely as a derivation of
the free commutative algebra $\mathcal{F}$, which then descends to the
quotient algebra $\mathcal{A}$.

The operator $\partial$ is a derivation with respect to the product,
and homogeneous of degree $-1$.

Let $\Delta$ be the operator $\exp(\partial)$. This is a well-defined
automorphism of $\mathcal{A}$. It does not preserve the grading, but
does preserve an increasing filtration by the number of inner edges.

Recall the $F$-triangle of a quadrangulation was defined in \ref{ftriangle} as
a polynomial in $x,y$.
\begin{proposition}
  The $F$-triangle is a morphism of commutative algebra from
  $\mathcal{A}$ to $\ZZ[x,y]$. It sends the derivation $\partial$ to
  $\partial_y$ and $\Delta$ to the substitution operator $y \mapsto y+1$.
\end{proposition}
\begin{proof}
  The fact that it is a morphism follows from
  Prop. \ref{simplicial_prod}. The derivation part follows from
  Prop. \ref{parabo_F}. The statement about $\Delta$ is a consequence
  of the statement about $\partial$.
\end{proof}

Let us call a \textit{cross-tree} the equivalence class of a
quadrangulation under the moves of twisting along an edge as defined
in \autoref{twisting}.

Cross trees can also be considered as abstract sets of squares, glued
together along their sides so as to form a tree-like shape, but where
the two possible gluings along an edge are not distinguished.

The notion of bridge introduced in \autoref{prop_bridge} still make
sense for cross-trees, as well as the associated notion of
connectedness.

By the theorem \ref{twisting_nests}, the set of serpent nests inside a
quadrangulation really only depends on the cross-tree.

One can define an algebra similar to $\mathcal{A}$ using cross-trees
instead of quadrangulations. This is a quotient algebra of $\mathcal{A}$.

If conjecture \ref{twisting_F} holds, the $F$-triangle only depends
on the cross tree, and therefore is defined on this quotient algebra.

\section{Enumerative aspects of quadrangulations}

\label{enum_section}

We shall compute here the number of connected quadrangulations up to
rotation and the number of connected cross-trees, up to isomorphism.

This is done using the classical language of species, see \cite{bergeron} for
the basics of this fundamental theory.

Let $X$ be the species of singletons. Let $E_2$ be the species of
unordered pairs. Let $C_4$ be the species of oriented cycles of length
$4$.

\subsection{Connected quadrangulations up to rotation}

Let $T$ be the species of all connected quadrangulations, $T^{\carre}$
the species of those pointed in a square, $T^{i}$ the species of those
pointed along an inner edge, and $T^{\carre,i}$ the species of those
pointed along a pair (square, adjacent inner edge).

Let $T^{b}$ be the species of quadrangulations that can be obtained
from a connected quadrangulation by cutting along an inner edge,
keeping one of the two halves and marking the cut edge. The species
$T^{b}$ is an auxiliary one, in which it is not allowed that the
square next to the marked edge has exactly one other neighbour opposite
to the edge, but where it is allowed that this square has only two
opposite neighbours. 

One has the following combinatorial relations:
\begin{align*}
  T^{\carre}+T^i&=T+T^{\carre,i},\\   
  T^i&=E_2(T^b),\\    
  T^{\carre,i}&=(T^b)^2,\\   
  T^{\carre}&=X (1+T^b+(T^b)^2+(T^b)^3 + C_4(T^b))),\\  
  T^b&=X (1+ 2 T^b+ 3 (T^b)^2 + (T^b)^3).   
\end{align*}
The first equation is an instance of the dissymetry equation for
trees, see \cite[Chap. 4]{bergeron} for the general principle behind
it. Other equations are just obvious recursive descriptions. The last
equation determines $T^b$ and the other ones in turn determines the
other species.

From these equations, one can compute the first few terms of the
associated generating series:
\begin{align*}
  T^b &: 1, 2, 7, 27, 114, 507, 2342, \dots\\  
  T   &: 1, 1, 1, 4, 11, 42, 155, 659, 2810, \dots\\  
  T^i &: 0, 1, 2, 10, 41, 196, 924,\dots\\
  T^\carre &: 1, 1, 3, 12, 52, 231, 1079,\dots\\  
  T^{\carre,i} &: 0, 1, 4, 18, 82, 385, 1848,\dots
\end{align*}


It is also interesting to consider the sub-species $T^{B}$ of $T^b$
where the square near the marked boundary edge does not have exactly
two opposite neighbours. In some sense, this condition is a stronger
form of connectedness. By removing one term in the equation defining
$T^{b}$, one gets that
\begin{equation}
  T^{B} = X (1 + 2 T^{b} + 2 (T^{b})^2 + (T^{b})^3),  
\end{equation}
which gives the first few terms
\begin{equation*}
  T^{B} : 1, 2, 6, 23, 96, 425, 1957, 9277, \dots  
\end{equation*}


\subsection{Connected cross-trees up to isomorphism}

Let us compute here the number of connected cross-trees.

Let $\T$ be the species of all connected cross-trees,
$\T^{\carre}$ the species of those pointed in a square, $\T^{i}$ the
species of those pointed along an inner edge, and $\T^{\carre,i}$ the
species of those pointed along a pair (square, adjacent inner edge).

Let $\T^{b}$ be the species of cross-trees that can be obtained from
a connected cross-tree by cutting along an inner edge,
keeping one of the two halves and marking the cut edge. The species
$\T^{b}$ is an auxiliary one, in which it is not allowed that the
square next to the marked edge has exactly one other neighbour opposite
to the edge, but where it is allowed that this square has only two
opposite neighbours.

One has the following combinatorial relations:
\begin{align*}
  \T^{\carre}+\T^i&=\T+\T^{\carre,i},\\
  \T^i&=E_2(\T^b),\\
  \T^{\carre,i}&=(\T^b)^2,\\
  \T^{\carre}&=X (1+\T^b+E_2(\T^b)+\T^b E_2(\T^b)+E_2(E_2(\T^b))),\\
  \T^b&=X (1+\T^b+(\T^b)^2+E_2(\T^b)+\T^b E_2(\T^b)).
\end{align*}
The first equation is again an instance of the dissymetry equation for
trees. Other equations are just obvious recursive descriptions. As
before, the last equation determines $\T^b$ and the other ones in turn
determines the other species.

From these equations, one can compute the first few terms of the
associated generating series for the number of isomorphism classes:
\begin{align*}
  \T^b &: 1, 1, 3, 7, 20, 58, 178, 557, \dots\\
  \T   &: 1,1,1,2,4,9,21,56,153,\dots\\
  \T^i &: 0,1,1,4,10,33,99,324,\dots\\
  \T^\carre &: 1,1,2,5,14,39,120,\dots\\
  \T^{\carre,i} &: 0,1,2,7,20,63,198,\dots
\end{align*}

Just as for quadrangulations, it is also interesting to consider the
sub-species $\T^{B}$ of $\T^b$ where the square near the marked boundary
edge does not have exactly two opposite neighbours. By removing one term in
the equation defining $\T^{b}$, one gets that
\begin{equation}
  \T^{B} = X (1 + \T^{b} + (\T^{b})^2 + \T^{b} E_2(\T^{b})),
\end{equation}
which gives the first few terms
\begin{equation*}
  \T^{B} : 1, 1, 2, 6, 16, 48, 145, 458, \dots
\end{equation*}


The first few terms of sequence $\T$ are illustrated in \autoref{taille5}.

\subsection{Connected quadrangulations with no cross}

For reasons related to fans (that are not considered in this article),
it is also interesting to count quadrangulations and cross-trees
avoiding the cross (which means that no square has four neighbours). We
will abuse notation by keeping the same names for these new species.

Let us start with quadrangulations. Using similar notations for
various species of connected quadrangulations avoiding the cross, one
has the following combinatorial relations:
\begin{align*}
  T^{\carre}+T^i&=T+T^{\carre,i},\\   
  T^i&=E_2(T^b),\\    
  T^{\carre,i}&=(T^b)^2,\\   
  T^{\carre}&=X (1+T^b+(T^b)^2+(T^b)^3),\\  
  T^b&=X (1+ 2 T^b+ 3 (T^b)^2).   
\end{align*}
The first equation is an instance of the dissymetry equation. Other
equations are just obvious recursive descriptions. The last equation
determines $T^b$ and the other ones in turn determines the other
species.

From these equations, one can compute the first few terms of the
associated generating series:
\begin{align*}
  T^b &: 1, 2, 7, 26, 106, 452, 1999, \dots\\  
  T   &: 1, 1, 1, 4, 10, 40, 141, \dots\\  
  T^i &: 0, 1, 2, 10, 40, 186, 846,\dots\\
  T^\carre &: 1, 1, 3, 12, 50, 219, 987,\dots\\  
  T^{\carre,i} &: 0, 1, 4, 18, 80, 365, 1692,\dots
\end{align*}

\subsection{Connected cross-trees with no cross}

Using the same notations as before, one has the following
combinatorial relations:
\begin{align*}
  \T^{\carre}+\T^i&=\T+\T^{\carre,i},\\
  \T^i&=E_2(\T^b),\\
  \T^{\carre,i}&=(\T^b)^2,\\
  \T^{\carre}&=X (1+\T^b+E_2(\T^b)+\T^b E_2(\T^b)),\\
  \T^b&=X (1+\T^b+(\T^b)^2+E_2(\T^b)).
\end{align*}
The first equation is an instance of the dissymetry equation for
trees. Other equations are just obvious recursive descriptions. The last
equation determines $\T^b$ and the other ones in turn determines the
other species.

From these equations, one can compute the first few terms of the
associated generating series for the number of isomorphism classes:
\begin{align*}
  \T^b &: 1, 1, 3, 6, 17, 44, 128, 365, 1091,\dots\\
  \T   &: 1, 1, 1, 2, 3, 8, 16, 42, 102,\dots\\
  \T^i &: 0, 1, 1, 4, 9, 29, 79, 244, 727,\dots\\
  \T^\carre &: 1, 1, 2, 5, 12, 34, 95, 280, 829,\dots\\
  \T^{\carre,i} &: 0, 1, 2, 7, 18, 55, 158, 482, 1454,\dots
\end{align*}

The reader can check the first few terms of sequence $\T$ in \autoref{taille5}.

\bibliographystyle{plain}
\bibliography{stokes_new}

\verb?chapoton@math.univ-lyon1.fr?\\
Frédéric Chapoton\\
Université de Lyon\\
CNRS UMR 5208\\
Université Lyon 1\\
Institut Camille Jordan\\
43 blvd. du 11 novembre 1918\\
F-69622 Villeurbanne cedex\\
France

\end{document}

%% file: tester_Q.pdf_tex
\begingroup%
  \makeatletter%
  \providecommand\color[2][]{%
    \errmessage{(Inkscape) Color is used for the text in Inkscape, but the package 'color.sty' is not loaded}%
    \renewcommand\color[2][]{}%
  }%
  \providecommand\transparent[1]{%
    \errmessage{(Inkscape) Transparency is used (non-zero) for the text in Inkscape, but the package 'transparent.sty' is not loaded}%
    \renewcommand\transparent[1]{}%
  }%
  \providecommand\rotatebox[2]{#2}%
  \ifx\svgwidth\undefined%
    \setlength{\unitlength}{420bp}%
    \ifx\svgscale\undefined%
      \relax%
    \else%
      \setlength{\unitlength}{\unitlength * \real{\svgscale}}%
    \fi%
  \else%
    \setlength{\unitlength}{\svgwidth}%
  \fi%
  \global\let\svgwidth\undefined%
  \global\let\svgscale\undefined%
  \makeatother%
  \begin{picture}(1,1)%
    \put(0,0){\includegraphics[width=\unitlength]{tester_Q.pdf}}%
    \put(0.01747918,0.6177){\color[rgb]{0,0,0}\makebox(0,0)[lb]{\smash{$i$}}}%
    \put(0.13782589,0.8310){\color[rgb]{0,0,0}\makebox(0,0)[lb]{\smash{$j$}}}%
    \put(0.35505149,0.9557){\color[rgb]{0,0,0}\makebox(0,0)[lb]{\smash{$k$}}}%
    \put(0.59624685,0.9557){\color[rgb]{0,0,0}\makebox(0,0)[lb]{\smash{$\ell$}}}%
  \end{picture}%
\endgroup%

%% file: tester_Q_bad.pdf_tex
\begingroup%
  \makeatletter%
  \providecommand\color[2][]{%
    \errmessage{(Inkscape) Color is used for the text in Inkscape, but the package 'color.sty' is not loaded}%
    \renewcommand\color[2][]{}%
  }%
  \providecommand\transparent[1]{%
    \errmessage{(Inkscape) Transparency is used (non-zero) for the text in Inkscape, but the package 'transparent.sty' is not loaded}%
    \renewcommand\transparent[1]{}%
  }%
  \providecommand\rotatebox[2]{#2}%
  \ifx\svgwidth\undefined%
    \setlength{\unitlength}{464bp}%
    \ifx\svgscale\undefined%
      \relax%
    \else%
      \setlength{\unitlength}{\unitlength * \real{\svgscale}}%
    \fi%
  \else%
    \setlength{\unitlength}{\svgwidth}%
  \fi%
  \global\let\svgwidth\undefined%
  \global\let\svgscale\undefined%
  \makeatother%
  \begin{picture}(1,0.98275862)%
    \put(0,0){\includegraphics[width=\unitlength]{tester_Q_bad.pdf}}%
    \put(0.01177823,0.56388253){\color[rgb]{0,0,0}\makebox(0,0)[lb]{\smash{$i$}}}%
    \put(0.1207268,0.75969551){\color[rgb]{0,0,0}\makebox(0,0)[lb]{\smash{$j$}}}%
    \put(0.31653979,0.87158866){\color[rgb]{0,0,0}\makebox(0,0)[lb]{\smash{$k$}}}%
    \put(0.55357661,0.88042234){\color[rgb]{0,0,0}\makebox(0,0)[lb]{\smash{$\ell$}}}%
    \put(0.87895003,0.44610027){\color[rgb]{0,0,0}\makebox(0,0)[lb]{\smash{$m$}}}%
  \end{picture}%
\endgroup%

%% file: essai_2_poly.tex
\tikzset{
  on each segment/.style={
    decorate,
    decoration={
      show path construction,
      moveto code={},
      lineto code={
        \path [#1]
        (\tikzinputsegmentfirst) -- (\tikzinputsegmentlast);
      },
      curveto code={
        \path [#1] (\tikzinputsegmentfirst)
        .. controls
        (\tikzinputsegmentsupporta) and (\tikzinputsegmentsupportb)
        ..
        (\tikzinputsegmentlast);
      },
      closepath code={
        \path [#1]
        (\tikzinputsegmentfirst) -- (\tikzinputsegmentlast);
      },
    },
  },
  mid arrow/.style={postaction={decorate,decoration={
        markings,
        mark=at position .5 with {\arrow[#1]{triangle 45}}
      }}},
}

\begin{tikzpicture}%
	[x={(0.488459cm, 0.782003cm)}, 
	y= {(-0.397415cm, 0.594354cm)}, 
	z= {(-0.776832cm, 0.187649cm)}, 
	scale=1.500000,
	back/.style={loosely dotted, thin},
	edge/.style={color=blue!95!black, thick,postaction={on each segment={mid arrow=black}}},
	facet/.style={fill=blue!95!white,fill opacity=0.100000},
	vertex/.style={inner sep=1pt,circle,draw=green!25!black,fill=green!75!black,thick,anchor=base}
        ]
%
%
\coordinate (-2.00, 1.00, -2.00) at (-2.00, 1.00, -2.00);
\coordinate (0.000, 3.00, 0.000) at (0.000, 3.00, 0.000);
\coordinate (-2.00, 1.00, 0.000) at (-2.00, 1.00, 0.000);
\coordinate (-2.00, 2.00, -3.00) at (-2.00, 2.00, -3.00);
\coordinate (0.000, 3.00, -3.00) at (0.000, 3.00, -3.00);
\coordinate (0.000, 2.00, -3.00) at (0.000, 2.00, -3.00);
\coordinate (0.000, 0.000, 0.000) at (0.000, 0.000, 0.000);
\coordinate (-2.00, 3.00, -3.00) at (-2.00, 3.00, -3.00);
\coordinate (0.000, 0.000, -1.00) at (0.000, 0.000, -1.00);
\coordinate (-1.00, 0.000, 0.000) at (-1.00, 0.000, 0.000);
\coordinate (-2.00, 3.00, 0.000) at (-2.00, 3.00, 0.000);
\coordinate (-1.00, 0.000, -1.00) at (-1.00, 0.000, -1.00);
\draw[edge,back] (-1.00, 0.000, -1.00) -- (-2.00, 1.00, -2.00);
\draw[edge,back] (0.000, 3.00, 0.000) -- (0.000, 0.000, 0.000);
\draw[edge,back] (-1.00, 0.000, 0.000) -- (-2.00, 1.00, 0.000);
\draw[edge,back] (0.000, 2.00, -3.00) -- (0.000, 0.000, -1.00);
\draw[edge,back] (0.000, 0.000, 0.000) -- (0.000, 0.000, -1.00);
\draw[edge,back] (0.000, 0.000, 0.000) -- (-1.00, 0.000, 0.000);
\draw[edge,back] (0.000, 0.000, -1.00) -- (-1.00, 0.000, -1.00);
\draw[edge,back] (-1.00, 0.000, 0.000) -- (-1.00, 0.000, -1.00);
\node[vertex] at (0.000, 0.000, 0.000)     {};
\node[vertex] at (0.000, 0.000, -1.00)     {};
\node[vertex] at (-1.00, 0.000, 0.000)     {};
\node[vertex] at (-1.00, 0.000, -1.00)     {};
\fill[facet] (-2.00, 3.00, 0.000) -- (0.000, 3.00, 0.000) -- (0.000, 3.00, -3.00) -- (-2.00, 3.00, -3.00) -- cycle {};
\fill[facet] (-2.00, 3.00, 0.000) -- (-2.00, 1.00, 0.000) -- (-2.00, 1.00, -2.00) -- (-2.00, 2.00, -3.00) -- (-2.00, 3.00, -3.00) -- cycle {};
\fill[facet] (0.000, 3.00, -3.00) -- (-2.00, 3.00, -3.00) -- (-2.00, 2.00, -3.00) -- (0.000, 2.00, -3.00) -- cycle {};
\draw[edge] (-2.00, 1.00, 0.00) -- (-2.00, 1.00, -2.000);
\draw[edge] (-2.00, 2.00, -3.00) -- (-2.00, 1.00, -2.00);
\draw[edge] (0.000, 3.00, 0.000) -- (0.000, 3.00, -3.00);
\draw[edge] (0.000, 3.00, 0.000) -- (-2.00, 3.00, 0.000);
\draw[edge] (-2.00, 3.00, 0.000) -- (-2.00, 1.00, 0.000);
\draw[edge] (0.00, 2.00, -3.00) -- (-2.000, 2.00, -3.00);
\draw[edge] (-2.00, 3.00, -3.00) -- (-2.00, 2.00, -3.00);
\draw[edge] (0.000, 3.00, -3.00) -- (0.000, 2.00, -3.00);
\draw[edge] (0.000, 3.00, -3.00) -- (-2.00, 3.00, -3.00);
\draw[edge] (-2.00, 3.00, 0.00) -- (-2.00, 3.00, -3.000);
\node[vertex] at (-2.00, 1.00, -2.00)     {};
\node[vertex] at (0.000, 3.00, 0.000)     {};
\node[vertex] at (-2.00, 1.00, 0.000)     {};
\node[vertex] at (-2.00, 2.00, -3.00)     {};
\node[vertex] at (0.000, 3.00, -3.00)     {};
\node[vertex] at (0.000, 2.00, -3.00)     {};
\node[vertex] at (-2.00, 3.00, -3.00)     {};
\node[vertex] at (-2.00, 3.00, 0.000)     {};
\end{tikzpicture}